\documentclass{amsart}

\usepackage{amsmath,amssymb,graphicx,epsfig}
\usepackage{amsmath}
\usepackage{amssymb}
\usepackage{framed}
\usepackage{graphics}
\usepackage{amsmath,amssymb,graphicx,epsfig}

\usepackage{amssymb,amsmath,amsfonts,amsthm,mathrsfs,graphicx,subfig}
\usepackage{wrapfig}
\usepackage[mac]{inputenc}
\usepackage{fancyhdr} 
\theoremstyle{break}
 \usepackage{xcolor}
 \usepackage{ulem}
 \usepackage{cancel}
\numberwithin{equation}{section}
 \newtheorem{thm}{Theorem}[section]
 \newtheorem{cor}[thm]{Corollary}
 \newtheorem{lem}[thm]{Lemma}
 \newtheorem{prop}[thm]{Proposition}
 \newtheorem{claim}[thm]{Claim}
 \newtheorem{defn}[thm]{Definition}
 \newtheorem{rem}[thm]{Remark}

\newtheorem*{thma}{THEOREM}

 \newcommand{\set}[1]{\left\{#1\right\}}

\newcommand{\eps}{\varepsilon}

\newcommand \be     {\begin{equation}}
\newcommand \ee     {\end{equation}}
\newcommand {\RR} {\mathbb{R}}

 \newcommand{\pa}{\partial}

 \newcommand{\suml}{\sum\limits}

 \newcommand{\ylambk}{y;\lambda,k}

 \usepackage{caption}
\subjclass[2010]{Primary 35J25; Secondary 35P20, 58J50}
\keywords{concentration, non-concentration, layered media, eigenfunctions, second-order elliptic, diffusion coefficient, piecewise constant, bounded variation, well of profile, exponential decay}
\date{\today}

\begin{document}
\setcounter{thm}{0}
\title{\textsf{\textbf{\Large
    CONCENTRATION AND NON-CONCENTRATION OF  EIGENFUNCTIONS OF SECOND-ORDER ELLIPTIC  OPERATORS IN LAYERED MEDIA}}}

\author{Assia Benabdallah$^{\dag}$, Matania Ben-Artzi$^{\ddag}$\,\& Yves Dermenjian$^{\dag}$,\\
\dag \,Aix Marseille Univ, CNRS, Centrale Marseille, I2M, Marseille, France\\
\ddag\, Institute of Mathematics, Hebrew University of Jerusalem, Jerusalem 91904, Israel\\
}

\begin{abstract}

  This work is concerned with  operators of the type $A=-\tilde{c}\Delta$  acting in domains $\Omega:=\Omega'\times (0,H)\subseteq\mathbb{R}^{d}\times\RR_+.$   The diffusion coefficient $\tilde{c}>0$ depends on one coordinate $y\in (0,H)$ and is bounded but may be discontinuous. This corresponds to the physical model of ``layered media'', appearing in acoustics, elasticity, optical fibers... Dirichlet boundary conditions are assumed. In  general, for each $\varepsilon>0,$ the set of eigenfunctions is divided  into a disjoint union of three subsets : $\mathfrak{F}_{NG}$ \,(non-guided), \,$\mathfrak{F}_{G}$\, (guided) \,and $\mathfrak{F}_{res}$ \,(residual). The residual set shrinks as $\varepsilon \to 0.$  The customary physical terminology of guided/non-guided is often replaced in the mathematical literature by concentrating/non-concentrating solutions, respectively.

  For guided waves, the assumption of ``layered media'' enables us to obtain rigorous  estimates  of their  exponential decay away from concentration zones. The case of non-guided  waves has attracted less attention in the literature. While it is not so closely connected to physical models, it leads to some very interesting questions concerning   oscillatory solutions and their asymptotic properties. Classical asymptotic methods are available for $c(y)\in C^2$ but a lesser degree of regularity excludes such methods. The associated eigenfunctions (in $\mathfrak{F}_{NG}$) are oscillatory. However, this fact by itself does not exclude the possibility of ``flattening out'' of the solution between two consecutive zeros, leading to concentration in the complementary segment. Here we show it cannot happen when $c(y)$ is of bounded variation, by proving a \textit{``minimal amplitude hypothesis''}. However the validity of such results when $c(y)$ is not of bounded variation (even if it is continuous) remains an open problem.
\end{abstract}
\maketitle
\section{\textbf{INTRODUCTION}}
\label{section-introduction}
%

Let $\Omega'\subseteq\mathbb{R}^{d},\,\,d=1,2,\ldots$ be an open bounded smooth domain. In particular, the eigenfunctions of $-\Delta$ in $\Omega'$ form a complete basis in $L^2(\Omega').$ Our domain of interest is
\be\label{eqomega}\Omega=\Omega'\times (0,H)\subseteq\mathbb{R}^{d}\times\RR_+.\ee
Observe that our regularity assumption on $\Omega'$ can be considerably relaxed, but this is not the main thrust of the present paper.

 In this paper one type of self-adjoint second-order elliptic operators is considered (details are given in Section ~\ref{secsetup} below)
 \be\label{eqoperator}
 \aligned
A=-\tilde{c}\,\Delta \,\,\mbox{in}\,\, L^{2}(\Omega;\tilde{c}(x)^{-1}\,dx) \,\,\mbox{with domain }\,\, H^2(\Omega)\cap H^1_0(\Omega).
\endaligned
\ee

Our study deals with \textit{layered media}, namely, the \textit{diffusion coefficient} $\tilde{c}$ depends only on the single spatial coordinate $y\in (0,H),$  so that $\tilde{c}(x)=\tilde{c}(x',y)=c(y). $ We use the terminology of \textit{diffusion coefficient} for lack of a better choice since it appears in the ``diffusive term''. Note that in the study of the associated wave equation it has the physical meaning of the variable speed of sound.

The dependence of $\tilde{c}$ on a single coordinate results in studying the spectral properties of $A$ via an infinite set of ordinary differential operators with effective increasing potentials (See Remark ~\ref{reminfinitekl}).

We always assume homogeneous Dirichlet boundary conditions. Generally speaking, the family of eigenfunctions is split into two categories: those sets of eigenfunctions (or sequences with increasing eigenvalues) involving concentration of  mass in proper subdomains of $\Omega,$ and those for which such concentration does not occur.

These two categories have been studied by physicists since a long time, investigating diverse phenomena ranging from acoustics to elasticity and to optical fibers\footnote{Optical fibers are associated with the Maxwell system and are a good illustration to the material in this paper, see \cite{Cherin:1}}. In general, the terminology used in the physical literature has referred to \textit{guided} or \textit{non-guided} waves, corresponding, respectively, to \textit{concentrating} or \textit{non-concentrating} modes. We shall use these terms interchangeably, as is appropriate in a particular context.

The reader is referred to ~\cite{grebenkov} for a survey of the geometrical structure of the eigenfunctions of the Laplacian, with very extensive bibliography.

As far back as 1930, Epstein \cite{EP:1}  established (in unbounded domains) the existence of acoustic guided waves that are generalized eigenfunctions, i.e. not belonging to the domain of the operator, and are evanescent outside a ``guiding channel''. The underlying speeds were analytic functions depending on a single vertical coordinate. See ~\cite{Ped-Wh:1} for a more general study of Epstein's profiles. An extensive study of guided waves in the acoustic case can be found in ~\cite{Wil:1} and its bibliography.

We mention briefly some other physical instances where guided waves play a significant role.
\begin{itemize}
\item The \textit{step-index fiber} ~\cite{Cherin:1}. It is a basic model of a cylindrical fiber consisting of a core and external shell (``cladding'') carrying
 two speeds with that of the core smaller than that of the shell. It  is a good example of the concentration of the energy in the core. This concentration is increasing when the  radius of the core is diminishing.

 \item In optoelectronics much attention is focused on the phenomenon of guided waves, governed by the Maxwell system ~\cite{Cherin:1,Marcuse}. In fact,  in the $``TE \,\,\& \,TM''$ framework the second-order equation for the amplitudes of eigenfunctions  ~\cite[Equation (13)]{BBH:1} is equivalent to our equation for the amplitude (see below Equation ~\eqref{equation-introduction:1}).

 \item The system of linear elasticity in the half-space $\Omega=\mathbb{R}^{n}\times(0,+\infty)$ subject to free surface condition. It gives rise to the  Rayleigh surface wave, that is particularly destructive in the case of an earthquake, see ~\cite{DerGui:1,Schu:1,Tolstoy:1}.  Related phenomena where studied by physicists such as Lamb, Love and Stoneley. Refer also to ~\cite{Cristofol:1} and references therein.
\end{itemize}

The terms \textit{concentration} and \textit{non-concentration} do not always carry the same meaning when used by various authors. The following definition clarifies their meanings in this paper.

For an open set $\omega\subseteq\Omega$ and $v\in L^2(\Omega),$ define

 \be\label{eqdefRw}R_{\omega}(v) = \frac{\Vert v\Vert^{2}_{L^{2}(\omega)}}{\Vert v\Vert^{2}_{L^{2}(\Omega)}}.\ee
\begin{defn}
\label{def-concentration:1}
If $\set{v_j}_{j=1}^\infty\subseteq L^2(\Omega)$ is a sequence of normalized eigenfunctions associated with an increasing sequence of eigenvalues and
\be\label{eqconcentr} \lim\limits_{ j\to\infty}R_{\omega}(v_j)=0\ee
then we say that $\set{v_j}_{j=1}^\infty$ \textbf{concentrates} in $\Omega\setminus\omega.$

On the other hand, if
\be\label{eqnonconcentr}
\liminf\limits_{ j\to\infty}R_{\omega}(v_j)>0,\quad \forall\omega\subseteq\Omega,
\ee
  then the sequence is \textbf{non-concentrating.}
 
\end{defn}
\begin{rem}
\label{rem-def-concentration:1}
    Later on we shall extend these notions also to sets of eigenfunctions that are not necessarily arranged as such sequences. Note that we study concentration and non-concentration for {\upshape{infinite subsets of eigenfunctions} } and not necessarily for the {\underline{whole set} } of eigenfunctions.
\end{rem}

In general, the occurrence of concentration phenomena for second-order  operators of the types ~\eqref{eqoperator} depends  on two features:
\begin{itemize}
\item The shape of the boundary $\pa\Omega.$
 \item The geometric properties of the diffusion coefficient  $\tilde{c}(x).$
 \end{itemize}
 The literature concerning the \textit{concentration/non-concentration phenomena} as related to the shape of $\Omega$ is very extensive. A well-known aspect is the connection of ``quantum ergodicity'' to ``classically chaotic systems'' ~\cite{BuZu:1,BuZw:1,Can-Galko:1,H-H-Marz:1,Marz:1} and references therein.  The paper ~\cite{grebenkov-1} deals with spherical and elliptical domains.

 In contrast,  in this paper we are interested in the effects of the layered medium.
   Thus it is more closely related to the study of operators of the type $L=-\nabla\cdot(c\nabla)+V$ on a finite domain, where the potential $V(x)\geq 0$ is positive on a subset of positive measure. Typically, eigenfunctions associated with eigenvalues below $ess\,sup\,V(x)$ are concentrating. In ~\cite{AFM:1} the authors replace $V$ by an \textit{effective potential} $u(x)$ satisfying $Lu=1.$ They show concentration and exponential decay of eigenfunctions as derived from the geometry of $u.$ Our operator $A$ ~\eqref{eqoperator} does not involve a potential but the concentration of suitable sequences of eigenfunctions results from the geometry of the diffusion coefficient. As we shall see in Theorem ~\ref{theo-general-c-:1} below there is a strong underlying geometric aspect;  the concentration expresses the fact that the masses of eigenfunctions ``flow'' (as the eigenvalues increase) into the ``wells'' (or ``valleys'').

 Turning to the {\underline{non-concentration case}}, we observe that the existing literature is less extensive, perhaps due to the fact that it is not directly related to physical or industrial applications. Nevertheless we shall see that  it leads to some interesting mathematical questions concerning the structure and asymptotics of  eigenfunctions (typically associated with large eigenvalues). Recent publications in this direction are ~\cite{H-Marz:1} dealing with non-concentration in partially rectangular billiards and ~\cite{Christ-Toth:1} concerning piecewise smooth planar domains.  A  non-concentration result in a stricter sense is that ``almost all eigenfunctions of a rational polygon are uniformly distributed'' ~\cite{rudnick}. Estimates for nodal sets such as ~\cite{DoFe:1}  were extended in ~\cite{JL:1,LaLe:1} motivated by questions from control theory and ~\cite{LLP:1} that deals with non-concentration in the Sturm-Liouville theory. Note that in the 1-D case issues of non-concentration are closely related to details of oscillatory solutions in the Sturm-Liouville theory. We shall come back to it later in this introduction.


This paper deals with  both  concentration and non-concentration phenomena  for eigenfunctions of layered operators. As already pointed out  the  latter is less studied in the literature, especially when the diffusion coefficient $c(y)$ is not regular (even discontinuous). As a result, the non-concentration case plays a greater role in this paper.  For such eigenfunctions we extend the scope of the study; not only facts pertaining to non-concentration but a more detailed study of the structure of the solutions in terms of the oscillatory character, amplitudes and their ratios and asymptotic behavior. In contrast to the concentrating case, we shall see that the essential features of the non-concentrating solutions depend primarily  on the maximum and minimum of $\tilde{c}(x) = \tilde{c}(x',y) = c(y)$ and, going deeper into the structures, on the \textit{total variation} of $c(y).$  Our main tool will be the  \textbf{minimal amplitude hypothesis} (see Definition ~\ref{definition-hypo-minam:1}), applied to \textit{families} of diffusion coefficients.

The paper is organized as follows.

In Section ~\ref{secsetup} we introduce all relevant notations and details concerning the functional setting.  In our case, the eigenvalues are classified by a double-index enumeration, with a conic sector (in index space {$(\mu_k^2,\,\lambda)$, see Figure \ref{image3}}) distinguishing eigenvalues (see ~\eqref{eqeigenconcentrate})  associated with concentrating eigenfunctions ($\mathfrak{F}_{G}$) from those (see ~\eqref{eqAceps})  associated with non-concentrating  eigenfunctions ($\mathfrak{F}_{NG}$). This curve serves as the analog to the maximal value of a perturbation potential that separates concentrating from non-concentrating eigenfunctions in the potential perturbation framework.

\begin{itemize}
\item Our main result for the concentrating case ($\mathfrak{F}_{G}$) is stated in  Theorem  ~\ref{theo-general-c-:1}.  In particular, it yields exponential decay of the eigenfunctions outside the concentration layers.




 \item In order to deal with  non-concentration of certain families of eigenfunctions ($\mathfrak{F}_{NG}$) we introduce
   the aforementioned \textbf{minimal amplitude hypothesis}.  This hypothesis is a geometric assumption on the asymptotic behavior  of the amplitudes in the $(u,\,u') $ phase plane.  The non-concentration of sets of oscillatory solutions follows directly from the geometric assumption (Theorem  ~\ref{theo-C1-noguided}).
   \end{itemize}

Section \ref{section-guidedwaves:1} deals with guided waves for $A = -\tilde{c}\Delta.$
 The main result Theorem \ref{theo-general-c-:1} is proved  and, on the way, we prove the existence of sequences of eigenvalues satisfying the hypotheses of this theorem (see condition ~\eqref{assumption:1}).  The exponential decay of eigenfunctions is derived from sharp estimates of the Green function.

  In Section ~\ref{section-nonguided waves:2} we turn to the case of non-concentrating eigenfunctions  (\textit{non-guided waves} in  the physical literature) for $A = -\tilde{c}\Delta.$ The set of corresponding eigenvalues is $ \mathscr{A}^{c}_\eps$ (see Definition ~\ref{defn-Ac}) that are located in the aforementioned upper conic sector in the index grid.

  The first approach that comes to mind is to transform the problem to a \textit{canonical form}. In other words, to use coordinate transformations so that the diffusion coefficient becomes a ``manageable'' perturbation of a constant one. In fact, this is done in Subsection ~\ref{section-Liouville-ng:1} under the assumption that $c(y)\in C^{2}(\lbrack 0,H\rbrack).$ In this case the classical Liouville transformation can be invoked, leading to a detailed asymptotic (almost sinusoidal) behavior of the non-concentrating eigenfunctions.

   Once the diffusion coefficient $c(y)$ is less regular, establishing non-concentration becomes considerably more delicate since the classical asymptotic methods are not applicable.  Thus, in the rest of Section ~\ref{section-nonguided waves:2} we focus on proving the minimal amplitude hypothesis that implies Theorem ~\ref{theo-C1-noguided}. Furthermore, the hypothesis is established simultaneously for a full family $\mathscr{K}$ of coefficients (see ~\eqref{eq-def-K}). It underlines the fact that only the extremal values of $c(y)$ come into play for Lipschitz continuous or monotone diffusion coefficients. We exploit  different methods in handling various classes of functions $c(y),$ such as Lipschitz functions in Subsubsection ~\ref{subsublip} or monotone functions in Subsubsection ~\ref{subsubmonotone}.  In each case, additional properties of the solutions are obtained, such as given in Corollary ~\ref{corunifcmonot} for the case of monotone coefficients. The ultimate case where we were able to establish the minimal amplitude hypothesis is for $c(y)$ being of \textbf{bounded total variation.} As a result non-concentration is shown to hold \textit{simultaneously} for the full family of diffusion coefficients of total variation $TV(c)$ below a fixed $V.$ More specifically we get

   \begin{thma}
   \textit{
    Fix $0<c_m<c_M,\,\,\eps>0,\,\,V>0.$ Let
      $$\mathscr{K}_V=\set{c(y),\,\,c_m\leq c(y)\leq c_M,\,\,0<y<H,\,\,\,TV(c)\leq V}.$$
            Consider (for every $c(y)\in\mathscr{K}_V$) the subset of eigenvalues $ \mathscr{A}^{c}_\eps$ (see \eqref{eqAceps}) and the associated eigenfunctions $\set{v_\lambda}.$
      \\
             For an interval $(a,b)\subseteq (0,H) $ let $\omega:=\omega'\times(a,b)\subseteq\Omega,$ where $\omega'\subseteq\Omega'$ is an open set.
             \\
             If $\omega'\neq\Omega'$ assume that the family $\set{\phi_k(x')}_{k=1}^\infty$ of eigenfunctions of the Laplacian in $\Omega'\subseteq\RR^d$ does not concentrate in $\Omega'\setminus\omega'.$
             \\
             Then there exists $\mathfrak{f}_\omega>0$ such that
      \be\label{eqnoncontrateab}
      \mathfrak{f}_\omega\leq\frac{\Vert v_{\lambda}\Vert_{L^{2}(\omega)}}{\Vert v_{\lambda}\Vert_{L^{2}(\Omega)}}\leq 1
      \ee
      uniformly for all $c(y)\in\mathscr{K}_V$ and all eigenvalues in $ \mathscr{A}^{c}_\eps.$
      }
      \end{thma}

      This theorem will be proved as part of the more detailed Theorem ~\ref{prop-minam:1-new}.
        \begin{rem}\label{rem-uniform}
        The uniformity statement in $\mathscr{K}_V$ is relevant for physical applications, where the coefficient $c(y)$ is only approximately known.
        \end{rem}

    Remark that   the case of a continuous $c(y),$  but  not of bounded variation,  remains an open problem, whence the following question arises naturally:
\begin{framed}
\textbf{What degree of regularity of $c(y)$ could serve as necessary and sufficient in order to satisfy the minimal amplitude hypothesis (Definition ~\ref{definition-hypo-minam:1})?}
\end{framed}

 As already mentioned, the model of piecewise constant coefficients is prevalent in the physical and engineering literature. We have therefore chosen to include  Section ~\ref{section-piecewiseconstant:1} where we treat in a \textit{self-contained way} the case of a piecewise constant diffusion coefficient $c(y)$ in both guided and non-guided cases.  In this treatment we implement more explicitly some tools that appear frequently in the physical literature, such as detailed expressions for the solutions in layers and their \textit{transmission relations} across layers.  In fact some estimates obtained here are sharper than those derived in Sections ~\ref{section-guidedwaves:1} and ~\ref{section-nonguided waves:2}.

%

%
%

    Appendix \ref{appendix-auxiliary-results-pwc:1} is added for auxiliary results.

In a subsequent paper we shall deal with the concentration and non-concentration issues for operators in divergence form.

\section{\textbf{SETUP AND MAIN RESULTS}}\label{secsetup}

Recall (~\eqref{eqomega}) that  $\Omega:= \Omega'\times(0,H).$  The coordinates in $\Omega$ are designated as $x=(x',y)\in\Omega'\times (0,H).$ We introduce a diffusion coefficient $(x',y)\to \tilde{c}(x',y)$ such that $\tilde{c}(x',y) = c(y)$ for all $x'\in \Omega'$ and of which we shall assume at least the following
 \begin{equation}\label{eqassumptioncy}
\mbox{\bf(H)}\quad 0<c(y)\in L^{\infty}(\lbrack 0,H\rbrack), 0<c_{m}={\rm ess\, inf}\{c(y), y\in \lbrack 0,H\rbrack\}<c_{M}= {\rm ess\,sup}\{c(y), y\in \lbrack 0,H\rbrack\}.
 \end{equation}

We focus on the operator $A= -\tilde{c}\,\Delta.$   For the Laplacian $-\Delta_{x'}$ acting in $L^{2}(\Omega')$ with domain $H^2(\Omega')\cap H^1_0(\Omega')$, we denote by $\set{(\mu_k^2,\phi_k)}_{k\geq1}$ the sequence of pairs (nondecreasing sequence of eigenvalues counting multiplicity, normalized eigenfunctions).  As the coefficient  function $\tilde{c}(x',y) = c(y)$ depends only on the last coordinate $y,$  a separation of coordinates  is natural. Using spectral decomposition in the $x'-$coordinate the operator $A:=-\tilde{c}\,\Delta$ is unitarily equivalent to a direct sum of reduced operators in the form
\begin{equation}
\label{equation-spectralstructure:1}\
-\tilde{c}\,\Delta\approx \sum_{k\in\mathbb{N}^*}\bigoplus A_{k}\quad \mbox{ acting in }\sum_{k\in \mathbb{N}^*}\bigoplus L^{2}((0,H), c(y)^{-1}dy)
\end{equation}
\be\label{eqAkofy}A_{k} :=c(y)\big( \mu_{k}^{2} - \frac{d^{2}}{dy^{2}}\big), k = 1,2,\ldots,\mbox{ with }D(A_k)=H^2(0,H)\cap H^1_0(0,H).
\ee
 The eigenvalues of $A$ are ordered by a two-index system, namely $\sigma(A)=\{\beta_{k,\ell}, k,\ell\geq 1\}$ where $\Lambda_{k} = \{\beta_{k,1},\beta_{k,2}, \ldots\}$ is the increasing sequence  of the eigenvalues of $A_{k}.$
In others words, for each eigenvalue  $\lambda$  of $A$ there exists at least one $k\in \mathbb N^*$ such that $\lambda$ is a simple eigenvalue of $A_k,$ whence there exists at least a pair $(k,\ell)\in \mathbb N^*\times \mathbb N^*$ such that $\lambda =\beta_{k,\ell}$ (there is a one-to-one  relationship between the pairs $(k,\ell)$ and $(\lambda, k)$).
 Note  that in general if $\lambda$ is not a simple eigenvalue, there is a finite number of pairs $(k,\ell)$ such that $\lambda=\beta_{k,\ell}.$

We construct an orthonormal basis  of eigenfunctions $\mathcal{B}=\set{v_{k,\ell}}_{k\geq 1,\ell\geq 1}$  associated with the eigenvalues $\beta_{k,\ell}.$  They are given by  $v_{k,\ell}(x',y)= \phi_{k}(x')u_{k,\ell}(y) $ where $u_{k,\ell}(y)$ satisfies
\begin{equation}
\label{equation-introduction:1}
 c(y)u_{k,\ell}'' + (\beta_{k,\ell}-c(y)\mu^{2}_{k})u_{k,\ell}= 0, \quad u_{k,\ell}(0) = u_{k,\ell}(H) = 0.
\end{equation}
\begin{rem}\label{reminfinitekl}
As is typical in ``separation of variables'' situations, the study of the spectral properties of the partial differential operator $A$ is carried out by controlling the behavior of the infinite set of ordinary differential operators of the type ~\eqref{equation-introduction:1}.
\end{rem}

\textbf{Henceforth we use the notation $u_{\lambda,k}$ instead of $u_{k,\ell}.$}
\\
 We often write $v_{\lambda}$ instead of $v_{k,\ell}.$
\begin{equation}
\label{equation-introduction:1bis}
\lambda = \beta_{k,\ell}\Longrightarrow v_{\lambda}(x',y)= v_{k,\ell}(x',y) = \phi_k(x')u_{\lambda,k}(y),\quad u_{\lambda,k}(y) \,\,\mbox{normalized in}\,\,L^{2}((0,H), c(y)^{-1}dy).
\end{equation}
In this paper we are primarily interested in the phenomena of concentration or non-concentration of the mass of eigenfunctions.
\begin{defn}
\label{defn-layer:1}
For $a<b,$ a layer of $\Omega$ will be noted $\Omega_{a,b}:=\Omega'\times( a,b)\subseteq\Omega.$
\end{defn}
 $\bullet$ {\textbf{ ON THE CONCENTRATION}}\\
 \vskip.5cm
\begin{minipage}[t]{168mm}
 \begin{wrapfigure}{r}{7cm}
\vskip-4.2cm
\includegraphics[scale=.33]{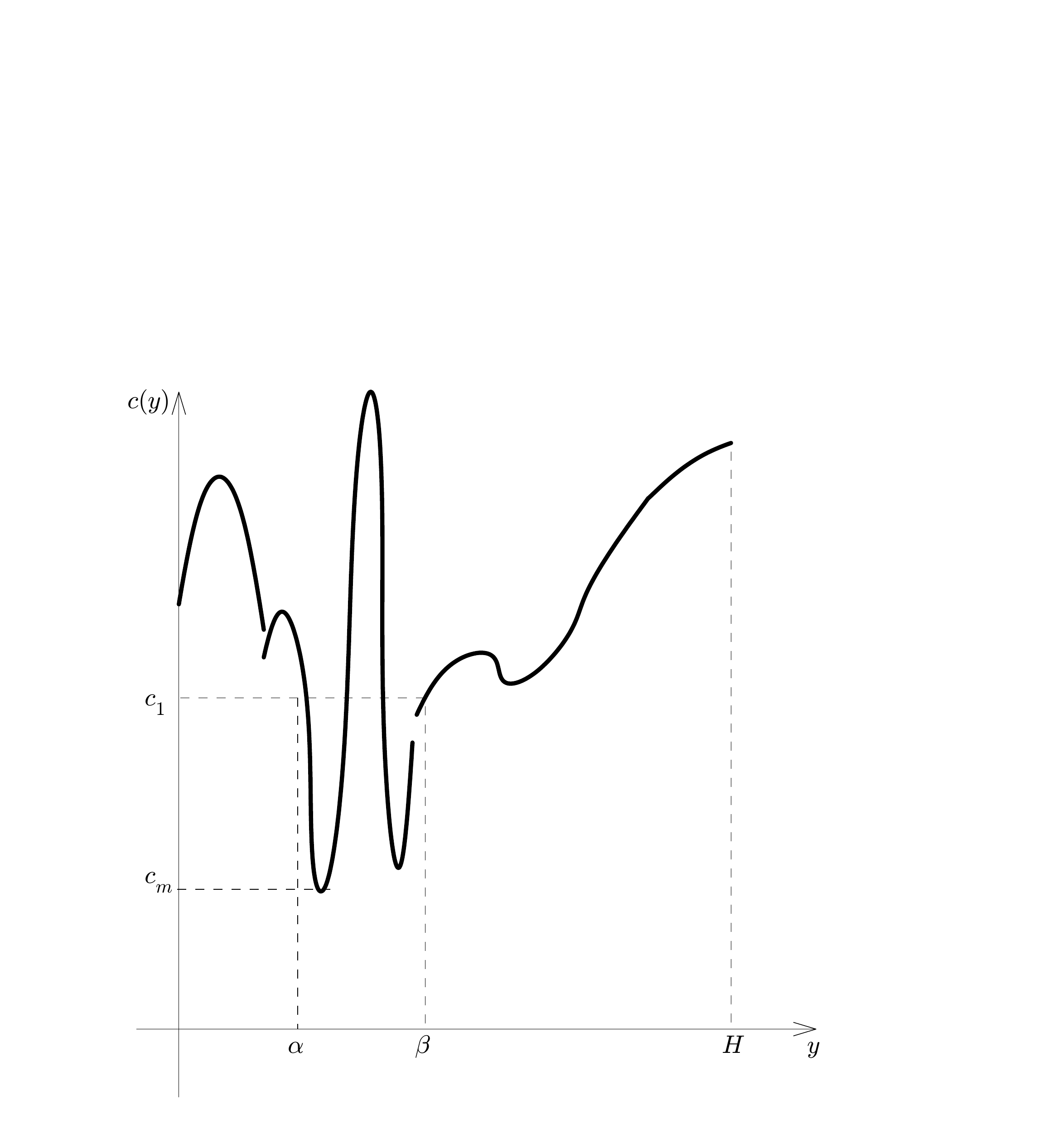}\vskip-1cm\caption{\label{image1}}
\end{wrapfigure}
\vskip-.5cm
\noindent
 \begin{defn}
\label{def-well:1}$\empty$\\
Let $\omega=\Omega_{\alpha,\beta}$ be a layer of $\Omega.$ We say that $\omega$ is a {\bf well for the profile} $c(y)$ if there exists $c_{1}>0$ such that
\begin{equation}
\label{cge:2}
\!\!\!\!\!\!\!\!\!\!\!\!\!\!\!\!\!\!\!\!\!\!\!\!\!\!\!\!\!\!\!\!\!\!\!\!\!\!\!\!\!\!\!\!\!\!\!\!\!\!\!\!\!\!\left\{\begin{array}{ll}
0<c_m
< c_1,\\
 c(y)\geq c_1>0,\quad a.e.\;y\in (0,H)\setminus (\alpha,\beta),
\end{array}
 \right.\end{equation}
 (See Figures ~\ref{image1} and \ref{image2}).
\end{defn}
\end{minipage}
\vskip1.2cm
  In the concentration case we have the following theorem, which yields exponential decay outside a well.   The proof is given in Section \ref{section-guidedwaves:1}. Observe that the only hypothesis imposed on $c(y)$ is \eqref{eqassumptioncy}.
\begin{thm}[Concentration in the layer $\Omega_{\alpha,\beta}.$]
\label{theo-general-c-:1}
 Let $\Omega_{\alpha,\beta}$ be a well for the profile $c$ and  $\lambda$ is an eigenvalue of  $A_{k}$ (hence of $A$) such that
  \be\label{assumption:1} \xi^{2} = \xi^{2}_{\lambda,k}:=  2(\mu_k^2-\frac \lambda{c_1})>0.\ee
  Let $\overline{\Omega_{a,b}}\subset\bar{\Omega}\setminus\overline{\Omega_{\alpha,\beta}}$ and denote by
 ${\rm d}$ the distance between $\overline{\Omega_{a,b}}$ and $\Omega_{\alpha,\beta}$, then
\begin{equation}\label{v8G}
\int_{\Omega_{a,b}}v_{\lambda}^2(x)\,dx\leq \underset{y\in(\alpha,\beta)}{\sup} \left\vert \frac{c_{1}-c(y)}{c_{1}c(y)}\right\vert \frac{\lambda}{\xi^3}  \frac{e^{-\xi (1-t){\rm d}}}{(1-t){\rm d}}\int_{\Omega_{\alpha,\beta}}v_{\lambda}^2(x)\,dx, \mbox{ for any }0<t<1.
\end{equation}
Consequently, if
\be\label{eqeigenconcentrate}c_m\,\mu_{k}^{2}\leq \lambda\leq (c_{1}-\varepsilon)\mu_{k}^{2}
\ee
 and if we normalize, i.e. $
\int_{\Omega} \vert v_{\lambda}(x)\vert^{2}\, \tilde{c}(x)^{-1}dx = 1,$  then
\begin{equation}
\label{equation-decreasing:1}
\underset{\lambda\to\infty}{\lim}\int_{\Omega_{a,b}}\vert v_{\lambda}(x)\vert^{2}\, dx = 0.
\end{equation}
\end{thm}
Observe that the assumption ~\eqref{eqeigenconcentrate} guarantees that $\xi\xrightarrow[\lambda\to\infty]{} \infty.$\\
\begin{figure}[h]
\centering
\includegraphics[scale=.25]{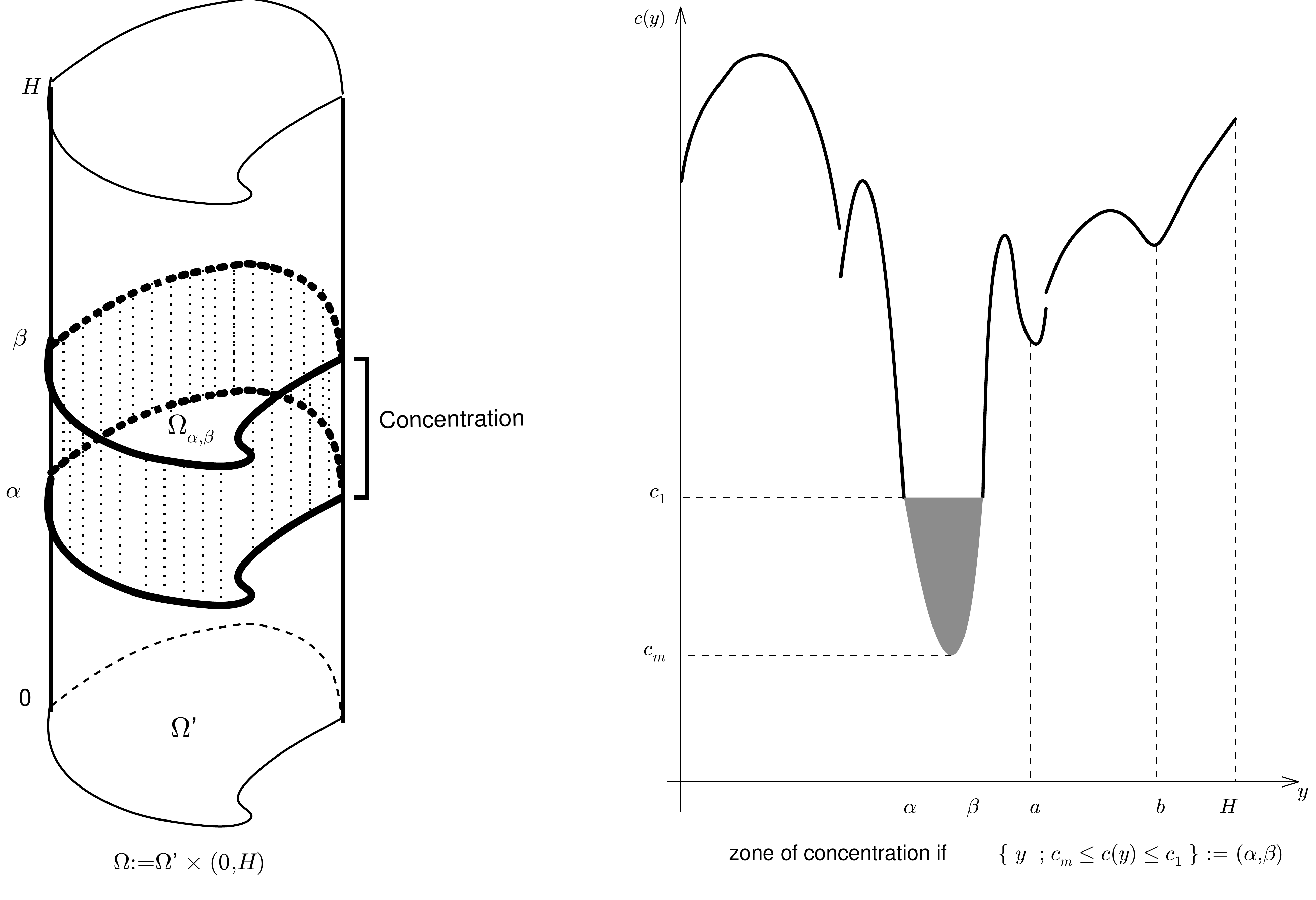}
\caption{}
\label{image2}
\end{figure}
The exponential decay in estimate~\eqref{v8G} can be compared to the results of ~\cite{AFM:1}. In our proof the 1-D dependence of $c(y)$ enables us to use sharp estimates of Green's kernel. On the other hand in ~\cite{AFM:1} the authors deal with a positive potential perturbation, that leads to a construction of an ``effective potential''. In terms of this potential the exponential decay is     expressed by an  ``Agmon-type'' \cite{AG:1} metric. In our case, from ~\eqref{equation-introduction:1} we can view the term $ c(y)\mu_{k}^{2}$ as the equivalent of a potential (but unbounded as $k\to\infty$).\\
\noindent
$\bullet$ {\textbf{ ON THE NON-CONCENTRATION.}}
 The second type of results concerns the  sets (indexed by $\varepsilon>0$) of non-guided normalized eigenfunctions (the set $\mathfrak{F}_{NG}$ of the Introduction). They are associated with eigenvalues
  $$\lambda= \beta_{k,\ell}> (c_{M}+\varepsilon)\mu_{k}^{2}, c_{M}:={\rm ess\,sup}_{y} c(y).$$
  This set is characterized
  by the fact that there is a positive lower bound for the masses in any layer $\Omega_{a,b},$  uniformly for all its elements.

     Recall  that $\lambda$ can correspond to several pairs $(k,\ell)$ and only some of them satisfy the above inequality.

  The geometrical interpretation of non-concentration is clear in the one-dimensional case $\Omega' = (0,L)$ and $\lambda= \beta_{k,\ell}> c_{M}\mu_{k}^{2}$: at each interface the angle between the wave and the normal is less than the critical angle stipulated by  geometric optics. So, the eigenfunction can travel across each layer without big loss.

  In physical applications it is conceivable that the diffusion coefficient $c(y)$ is known only approximately. It is therefore interesting to extend our study to deal with sets of such coefficients.  Let $0<c_m<c_M$ be fixed. We assume that every
 coefficient $c(y)$ satisfies condition $\mbox{\bf(H)}$ (see ~\eqref{eqassumptioncy}) and denote by
  \be\label{eq-def-K}\mathscr{K}=\set{c(y),\quad c_m={\rm ess\,inf}_{y} c(y)<c_M={\rm ess\,sup}_{y} c(y)}\ee
  the family of all such coefficients.

In various cases, we shall impose further assumptions on the elements of $\mathscr{K}.$

  We introduce the set of eigenvalues as above,  whose associated eigenfunctions will be shown to be (perhaps under additional assumptions) non-concentrating.
\begin{defn}\label{defn-Ac} Fix $\eps>0.$ For any fixed $\mu_k,$ let $\ell_{0,k}$ be the first $\ell$ satisfying $\beta_{k,\ell_{0,k}}\geq (c_M+\eps)\mu_k^2.$\\
We designate
 (see Figure \ref{image3})
\be\label{eqAceps} \mathscr{A}^{c}_\eps={\underset{k=1}{\overset{\infty}\bigcup}}\set{(\mu_k,\lambda),\,\,\,\lambda= \beta_{k,\ell},\quad\ell\geq\ell_{0,k}}.
\ee
\end{defn}
Next we define the \textbf{minimal amplitude} of the family of the associated solutions as follows.
\begin{equation}
 \label{equation-definition-hypo-minam:1}
 \mathfrak{r}_c^{2} = \inf_{\begin{array}{c} y\in \lbrack 0, H\rbrack,\\ (\mu_k,\lambda)\in\mathscr{A}^{c}_\eps\end{array}}\lbrack u_{\lambda,k}(y)^{2} + u'_{\lambda,k}(y)^{2}\rbrack.
 \end{equation}
 In the subsequent discussion the parameter $\eps>0$ is fixed and for simplicity of notation we omit the indication of the  dependence of $\mathfrak{r}_c$ on it.

 \begin{defn}
 \label{definition-hypo-minam:1}
 Let $c(y)\in \mathscr{K}.$ We say that $c(y)$ satisfies the {\textbf{\textit minimal amplitude hypothesis}} with respect to $\mathscr{A}^{c}_\eps$  if
 \be\label{eqrcgreat0}\mathfrak{r}_c>0.\ee
 \end{defn}
\begin{rem}\label{remprufer}
 Note that this hypothesis has a very clear geometric interpretation by means of the Pr\"{u}fer substitution ~\cite{birkhoff}.
\end{rem}
 \noindent
Observe that while the minimal amplitude deals with the sum of squares $u_{\lambda,k}(y)^{2} + u'_{\lambda,k}(y)^{2},$ the non concentration involves only the integral of $u_{\lambda,k}(y)^{2}$ over various intervals. The following Theorem ~\ref{theo-C1-noguided} connects these topics, showing that the minimal amplitude hypothesis implies non-concentration. Here we state it using the physical model with the spectral parameter $\lambda.$ It is  proved in a somewhat more detailed form (using the reduced eigenfunctions $u_{\lambda,k}$) as Theorem ~\ref{propcyirreg}  in Subsection ~\ref{subsection-sufficient condition-ng}.
\begin{thm}[Non-concentration in any layer]
\label{theo-C1-noguided}
Let $c(y)\in \mathscr{K}$ be a diffusion coefficient satisfying the minimal amplitude hypothesis. For any $(\mu_{k},\lambda)\in \mathscr{A}^{c}_\eps$  let $v_{\lambda}(x)$  be an associated eigenfunction.
  \\
 Let $\omega:=\omega'\times(a,b)\subseteq\Omega$ be an open set.
  If $\omega'\neq\Omega'$ assume that the family $\set{\phi_k(x')}_{k=1}^\infty$ of eigenfunctions of the Laplacian in $\Omega'\subseteq\RR^d$ does not concentrate in $\Omega'\setminus\omega'$ (see Definition \ref{def-concentration:1} and Remark \ref{rem-def-concentration:1}).
  \\
  Then  there exists  a constant $C_{\omega}>0$  such that,
\begin{equation}
\label{ineq-ng-general:1}
0< C_{\omega}\leq \inf\limits_ {(\mu_{k},\lambda)\in \mathscr{A}^{c}_\eps}\,\frac{\Vert v_{\lambda}\Vert_{L^{2}(\omega)}}{\Vert v_{\lambda}\Vert_{L^{2}(\Omega)}}\leq 1.
\end{equation}
\end{thm}
\begin{rem}\label{remuniformcy}
We shall see that in various cases we can find subsets $\mathscr{K}_1\subseteq\mathscr{K}$ such that the inequality ~\eqref{ineq-ng-general:1} holds uniformly with respect to $c\in \mathscr{K}_1.$
\end{rem}
\setcounter{figure}{2}
\hspace{-.5cm}\begin{minipage}[t]{169mm}
\begin{wrapfigure}{r}{6.1cm}
\vskip-0.3cm
\includegraphics[scale=.3]{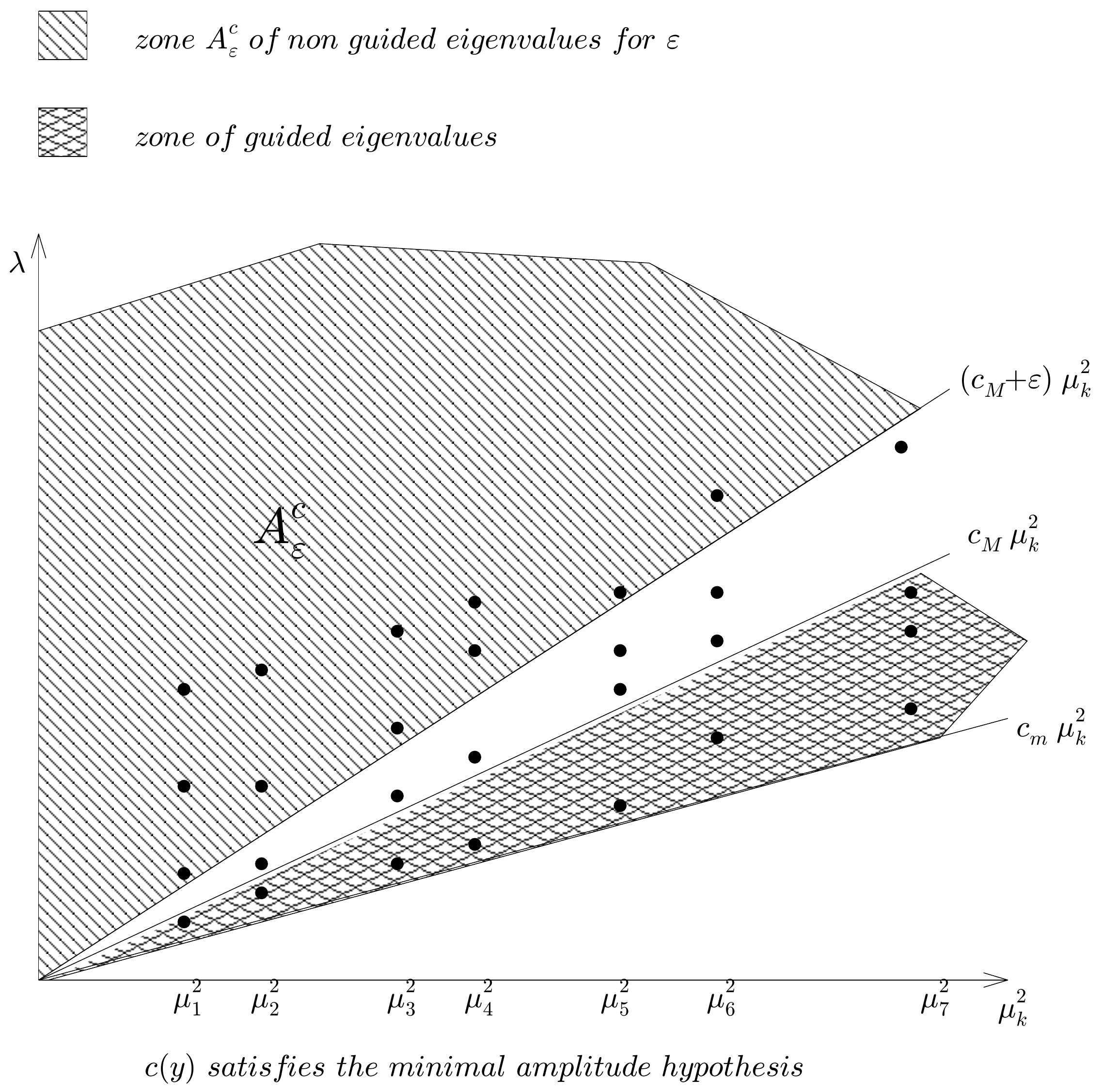}\vskip-.4cm\caption{\label{image3}}\hfill
\end{wrapfigure}
\noindent
\begin{rem}\label{rem-about-min-amplitude}$\empty$
 \begin{itemize}
 \item
  In Section ~\ref{section-nonguided waves:2} we show that any eigenfunction associated with eigenvalues in $\mathscr{A}^{c}_\eps$ behaves in an oscillatory fashion. This is a straightforward consequence of the comparison principle. However, it does not exclude the possibility that some of the sections of the oscillatory solution may ``flatten out'', namely their amplitudes shrink as $\lambda\to\infty.$
  The condition ~\eqref{equation-definition-hypo-minam:1} ensures that such phenomena do not happen, as is stated in Theorem ~\ref{theo-C1-noguided}.
\item
 In Subsection ~\ref{subsection-sufficient condition-ng} we discuss the meaning of the minimal amplitude hypothesis. If the function $c(y)$ is of bounded total variation we prove (Theorem ~\ref{prop-minam:1-new}) that it satisfies the hypothesis with respect to $ \mathscr{A}^{c}_\eps.$  This covers the cases of
 functions in $ C^{1}(\lbrack 0,H\rbrack)$ as well as  functions in $W^{1,1}(\lbrack 0,H\rbrack)$,  piecewise constant functions ...
 \end{itemize}
 \end{rem}
 \end{minipage}
\vskip.2cm
\noindent
\section{\textbf{GUIDED WAVES (SEE THEOREM ~\ref{theo-general-c-:1})}
 }
 \label{section-guidedwaves:1}
We refer to the geometric setup in Definition \ref{def-well:1} above. The simplified case with $\alpha = 0, \beta =h_{0}$  is common in the physical literature dealing with band structure. It was our starting point at the early stage of this work \cite{BBD:3}. In this section we always assume ~\eqref{eqassumptioncy}.
\subsection{\textbf{PROOF OF THEOREM ~\ref{theo-general-c-:1}}}$\empty$\label{subsecproof}\\
Using the notation in Definition  ~\ref{def-well:1} and ~\eqref{equation-introduction:1bis} we are interested in the behavior of $u_{\lambda,k}$ (solution to \eqref{equation-introduction:1}) as $\lambda\to \infty$. Note that for all $(\lambda,k), \,\lambda = \beta_{k,\ell},$ the function $w_{\lambda,k}(y):=u_{\lambda,k}(y)^2$ is a solution to (where we use $w$ for simplicity)
 \begin{eqnarray*}\left\{\begin{array}{ll}w''(y) -2(\mu_k^2-\frac{\lambda}{ c(y)})w(y)=g(y)\mbox{ on } (0,H),\\
w(0)=w'(0)=w(H)=w'(H)=0\,\end{array}\right.
\end{eqnarray*}
where $g=2(u'_{\lambda,k})^2$.
%
The introduction of $\xi=\sqrt{2(\mu_k^2-\frac \lambda{c_1})}$ transforms the previous equation into
\begin{eqnarray}\label{pbg}\left\{\begin{array}{ll}w'' -\xi^2w=-f+ g(y)\mbox{ on }(0,H),\\
w(0)=w'(0)=w(H)=w'(H)=0,\end{array}\right.
\end{eqnarray}
with
\begin{equation}\label{fg}
f(y)= 2\lambda\frac{c_1-c(y)}{c_1c(y)} w(y),\quad y\in  (0,H).
\end{equation}
\textbf{Some properties of the solutions of ~\eqref{pbg}}
\label{subsection-properties-solutions-general case:1}
\\
 We now derive upper and lower bounds for solutions of ~\eqref{pbg}.
\begin{claim}[Upper pointwise bounds for the geometric situation as in Definition \ref{def-well:1}]
	\label{eb}
Let $w$ be a solution to \eqref{pbg}-\eqref{fg}. Let $d(y,(\alpha,\beta))$ be the distance of $y$ to the interval $(\alpha, \beta).$ Then, if $c_{m}\,\mu_{k}^{2}\leq \lambda <c_{1}\mu_k^2, $ we have
 \begin{eqnarray}\label{vG}
y\in (0,H)\setminus (\alpha,\beta)
 &\!\!\!\Longrightarrow &\!\!\!w( y)\leq \underset{y'\in(\alpha,\beta)}{\sup} \left\vert \frac{c_{1}-c(y')}{c_{1}c(y')}\right\vert  \frac \lambda \xi e^{-\xi d(y,(\alpha,\beta))}\int_\alpha^\beta w(z)\,dz.
\end{eqnarray}
\end{claim}
\begin{proof} For simplicity of the presentation we take $\alpha=0.$
We use the Green function $G$ of the Dirichlet operator $w''-\xi^{2}w$ and prove exponential decay outside the well if $\lambda< c_{1}\mu_{k}^{2},$   depending on the distance of $y$ to the well.
\\
The Green function is given by
\begin{equation}\label{green}
\forall y,y'\in (0,H),\quad G(y,y';\xi): =\left\{\begin{array}{lll}-\frac {\sinh(\xi y)}{\xi \sinh(\xi H)}\sinh(\xi(H-y')),\quad y<y',\\
-\frac {\sinh(\xi y')}{\xi \sinh(\xi H)}\sinh(\xi(H-y)),\quad y>y'.
\end{array}\right.
\end{equation}
Then, from \eqref{pbg},
$\forall y\in (0,H),\; w(y)= -\int_0^HG(y,y';\xi)\,f(y')\,dy'+ 2\int_0^HG(y,y';\xi)(u'_{k,\lambda})^2(y')\,dy'$ and, as $G\leq 0$, we get
\begin{equation*}
\label{v1g}
\forall y\in (0,H),\quad w(y)  \leq  -\int_0^HG(y,y';\xi)\,f(y')\,dy'.
\end{equation*}
In view of ~\eqref{cge:2} one has $c(y)\geq c_{1}$ outside $(0,\beta)$ and, as $f$ is nonpositive on $(0,H)\setminus(0,\beta)$ and $G\leq 0,$ this implies
\begin{equation}
\label{v2G}
\forall y\in (0,H)\setminus(0,\beta),\quad w(y)  \leq  -\int_0^\beta G(y,y';\xi)\,f(y')\,dy'.
\end{equation}
The estimate \eqref{v2G} becomes
$w(y)\leq -2\lambda  \int_0  ^\beta \big(\underset{y'\in(0,\beta)}{\sup} \vert \frac{c_{1}-c(y')}{c_{1}c(y')}\vert\big) G(y,y';\xi) w(y')\,dy'$ since
$-G\geq 0$ and $w(y)\geq 0.$ Taking $y\in (0,H)\setminus (0,\beta),$
 the definition of $G$ applied to \eqref{v2G} gives
$$w(y)\leq 2\frac \lambda \xi \underset{y'\in(0,\beta)}{\sup} \left\vert \frac{c_{1}-c(y')}{c_{1}c(y')}\right\vert\frac{\sinh(\xi \big(H-y\big))}{\sinh(\xi H)}\int_0  ^\beta \sinh(\xi y')w(y')\,dy' .
$$
The estimate~\eqref{vG} is deduced from
$\frac {\sinh(\xi (H-y))}{\sinh(\xi H)}\leq e^{-\xi y},\;\underset{[0,\beta]}\sup \sinh(\xi y')\leq \frac{1}{2} e^{\xi \beta},$ since the distance of $y$ to $(0,\beta)$ is $d(y,(0,\beta)) =y-\beta$.
\end{proof}
\begin{claim}[Lower pointwise bounds for the solution]
\label{prop2G}
Let   $\lbrack a,b]\subset \lbrack0, \alpha)\cup(\beta,H]$ and $d = d((\alpha,\beta),(a,b)).$ Then, for all $ t\in (0,1)$, any $w$  solution of \eqref{pbg} verifies
\begin{eqnarray}
\label{v4G}
\begin{array}{ll}  a>\beta\,\,\Rightarrow  w( t a+(1-t)\beta)\geq \frac{1}{2}\xi^2 (1-t)d \int_a^b w(y)\,dy,\\
   b<\alpha\,\,\Rightarrow w (t b+(1-t)\alpha)\geq   \frac{1}{2} \xi^2  (1-t)d \int_a^b w(y)\,dy.
   \end{array}\end{eqnarray}
   \end{claim}
\begin{proof}
Again for  simplicity we take $\alpha = 0.$ Taking into account that $w(H)=w'(H)=0$, if $a>\beta$, we integrate twice \eqref{pbg} from $H$.  As $(u'_{k,\lambda})^2\geq0$  and $f\leq 0$ on $(\beta,H)$, we obtain with $\gamma = ta + (1-t)\beta, 0<t<1,$
 \begin{equation*}
 \label{equation-lowerbound:1}
 \forall a>\beta,\quad w (\gamma)\geq  \int_{\gamma}^H  \int_y^H \xi^{2}
w(z)\,dz\,dy =  \xi^2 \int_{\gamma}^H (y-\gamma)w(y)\,dy\geq \xi^{2}(a-\gamma)\int_{a}^{b}w(y)\,dy.
\end{equation*}
\end{proof}
 \noindent
\textbf{Conclusion of the proof of Theorem \ref{theo-general-c-:1}:}
The estimates ~\eqref{v8G}-\eqref{equation-decreasing:1} follow directly by combining \eqref{vG} and \eqref{v4G}.
\begin{rem}
In the estimates above we have used the explicit form of the Green kernel. As an alternative we could use general trace estimates that are applicable also for divergence-type operators $-\nabla\cdot(\tilde{c}\nabla),$ where an explicit kernel is not available. However, this method yields only a polynomial rate of decay $O\left(\frac{1}{\mu_{k}^{2(1-s)}}\right)$ in \eqref{v8G}. This approach will be used in a subsequent paper.
\end{rem}
\noindent
Observe that if the profile $c(y)$ has two wells, with the same "depth" $c_m$ (see Definition \ref{def-well:1}), then the method of proof of Theorem \ref{theo-general-c-:1} fails. However, by enlarging $(\alpha,\beta)$ so that it contains the two wells, we can repeat the proof to get concentration in this extended band.

\begin{rem}[\textbf{Estimating in terms of a subdomain of the well}]\label{remsubwell}
    Note that in the right-hand side of the estimate ~\eqref{v8G} the mass in the well is
    $$\int_{\Omega_{\alpha,\beta}}v_{\lambda}^2(x)\,dx=\int_\alpha^\beta u_{k,l}(\ylambk)^2dy\cdot\int_{\Omega'}|\phi_k(x')|^2dx'.$$
    Suppose that there exists an open domain $\omega'\subseteq\Omega'$ and a subsequence (retaining the same index) $\set{\phi_k(x')}_{k=1}^\infty$ that does not concentrate in $\Omega'\setminus\omega'.$ Then clearly $\int_{\Omega_{\alpha,\beta}}v_{\lambda}^2(x)\,dx$ can be replaced by $\int_{\omega'\times(\alpha,\beta)}v_{\lambda}^2(x)\,dx.$
\end{rem}

\subsection{\textbf{EXISTENCE OF EIGENVALUES COMPATIBLE WITH ASSUMPTION } ~\eqref{assumption:1}}$\empty$\\
\label{subsection-existence-eigenvalues-general case:1}
The previous results rely on the existence of infinitely many eigenvalues satisfying Assumption ~\eqref{assumption:1}. This fact is established in the
following theorem.
\begin{thm}
\label{theo-existence-eigenvalues-general case:1}
The number of eigenvalues satisfying Assumption ~\eqref{assumption:1} goes to infinity with $k.$
\end{thm}
\begin{proof}
We use three ingredients. Let $\alpha,\,\beta$ be as in ~\eqref{cge:2}.  First, the function $\lbrack 0,H\rbrack\ni y\mapsto\int_{c(y')\leq y} {\rm d}y'$ being continuous there exists a nonempty open set $U\subseteq (\alpha, \beta)$ and $\epsilon_{0}>0$ satisfying $c_{m} \leq c(y)<c_{1}-2\varepsilon_{0}, y\in U.$ Second, for each $k\in \mathbb{N^*}$ the smallest eigenvalue $\beta_{k,1}$ of $A_{k}$ is given by  $\underset{w\in H^{1}_{0}(0,H)}{\inf} \frac{\int_{0}^{H}(\vert w'(y)\vert^{2} + \mu_{k}^{2}\vert w(y)\vert^{2}) {\rm d}y}{\int_{0}^{H}c^{-1}(y) \vert w(y)\vert^{2}) {\rm d}y}.$ Third, we know that $\underset{w\in H_{0}^{1}(U)}{\inf}\frac{\int_{U}\vert w'(y)\vert^{2}{\rm d}y}{\int_{U}\vert w(y)\vert^{2}{\rm d}y}$ is the smallest eigenvalue $\theta_{1}>0$ of the operator  $-\frac{d^{2}}{dy^{2}}$ defined on  $U$ with  Dirichlet boundary  conditions.
 So, we can write
 \begin{eqnarray}
 \label{equation-existencevp:1}
 \beta_{k,1}& \leq & \underset{w\in H_{0}^{1}(U)}{\inf}\frac{\int_{U}(\vert w'(y)\vert^{2} + \mu_{k}^{2}\vert w(y)\vert^{2}) {\rm d}y }{\int_{U}c^{-1}(y) \vert w(y)\vert^{2}) {\rm d}y}\nonumber\\
 & \leq & (c_{1}-2\varepsilon_{0}) \underset{w\in H_{0}^{1}(U)}{\inf}\frac{\int_{U}(\vert w'(y)\vert^{2} + \mu_{k}^{2}\vert w(y)\vert^{2}) {\rm d}y }{\int_{U} \vert w(y)\vert^{2} {\rm d}y}\nonumber\\
 & \leq & (c_{1}-2\varepsilon_{0}) (\mu^{2}_{k} + \theta_{1}).
 \end{eqnarray}
 Take $K>0$ sufficiently large, so that $\varepsilon_{0}\mu_{k}^{2}>(c_{1}-2\varepsilon_{0})\theta_{1}$ when $k>K.$ Then from \eqref{equation-existencevp:1} we get
 \begin{equation*}
 \beta_{k,1} < (c_{1}-\varepsilon_{0})\mu^{2}_{k}, \;\mbox{ for } k>K.
 \end{equation*}
 Then the sequence $(\beta_{k,1})_{k}$ satisfies Assumption ~\eqref{assumption:1} for $k>K.$ This proof exhibits only a sequence but we can build other sequences satisfying this assumption. We skip a detailed discussion of this fact for the sake of brevity.
 \end{proof}
\section{\textbf{NON-GUIDED WAVES }}
\label{section-nonguided waves:2}

\noindent
  An (infinite) set of non-guided normalized eigenfunctions  is characterized
  by the fact that in each layer $\Omega_{a,b}$ there is a uniform positive lower bound for the masses in the layer, valid for all  elements of the set.

 As observed in the introduction, for each eigenvalue $\lambda$ of $A,$ there exists at least one pair $(k,\ell)$ so that $\lambda=\beta_{k,\ell}$ is the $\ell$-th eigenvalue of $A_{k}.$
  Let $\lambda = \beta_{k,\ell}>0$ be an eigenvalue of $-c(y)\Delta$ in $L^2(\Omega,c(y)^{-1}dx'dy)$ and $u(y;\lambda,k):=u_{\lambda, k}(y)$ the normalized  associated (reduced) eigenfunction (as in \eqref{equation-introduction:1} and \eqref{equation-introduction:1bis}). The function $u(y;\lambda,k)$ satisfies
        \begin{eqnarray}
        \label{eqnonguide}
       u''(\ylambk)+Q(y;\lambda,k)u(\ylambk) &= &u''(\ylambk)+\mu^2_{k}\Big(\frac{\lambda}{\mu^2_{k}c(y)}-1\Big)u(\ylambk),
       \\
       & = & u''(\ylambk) + \mu_{k}^{2}\, p(\ylambk)u(\ylambk) = 0,\nonumber\\
       &p(\ylambk)= &\frac{\lambda}{\mu^2_{k}c(y)}-1.\nonumber
       \end{eqnarray}
        $ u(0;\lambda,k)=u(H;\lambda,k)=0$ and $\int_0^H u(y;\lambda,k)^2c(y)^{-1}dy=1.$

 We shall deal in this section with eigenvalues $\lambda$ such that
 (see ~\eqref{eqAceps})
$$(\mu_k,\lambda)\in  \mathscr{A}^{c}_\eps.$$
In particular, for such values we have $p(\ylambk)>0.$

A desirable way to treat this equation is by transforming the equation into a canonical equation of the type
\begin{equation*}
\frac{d^{2}\eta}{d\xi^{2}} + \mu_{k}^{2}\eta	 = \rho(\xi;\lambda,k)\eta
\end{equation*}
with some new variable $\xi$ and new unknown $\eta.$\\
This classical procedure, known as the {\it Liouville transformation}, can be carried out only if $c(y)$ is twice continuously differentiable, and is used in Subsection \ref{section-Liouville-ng:1} when $c\in C^{2}(\lbrack 0,H\rbrack),  (c_{M}+\varepsilon)\mu_{k}^{2}<\lambda<(c_{M}+\Lambda)\mu_{k}^{2}, \,0<\varepsilon <\Lambda.$

 The aim of the subsequent subsections is to claim that the set of eigenfunctions
associated with eigenvalues satisfying $(\mu_k,\lambda)\in\mathscr{A}^{c}_\eps ,$
for any diffusion coefficient $c(y)\in\mathscr{K}$ (see ~\eqref{eq-def-K})
consists of non-guided eigenfunctions when a particular sufficient condition is satisfied, with $c(y)$ less regular.


\noindent\\

Consider a pair $(\mu_k,\lambda)\in\mathscr{A}^{c}_\eps .$
Let $u(\ylambk)$ be a normalized solution  to \eqref{eqnonguide}, associated with $(\mu_{k}, \lambda).$

In view of ~\eqref{eqAceps}

\begin{equation}
\label{ineq-intermedzone:2}
p(\ylambk) \geq \frac{\eps}{c_M}, \quad y\in \lbrack 0, H\rbrack.
\end{equation}

Let
\begin{equation*}
Z(\lambda,k) = \{z_{0} =0<z_{1}<z_{2}<\ldots < z_{ \mathfrak{s}} = H, u(z_{i};\lambda, k) = 0, \quad 0\leq  i\leq \mathfrak{s}\}\subseteq \lbrack 0, H\rbrack
\end{equation*}
be the set of zeros of the function $u(\ylambk).$ It follows from \eqref{ineq-intermedzone:2} and the comparison principle ~\cite[Section X.6]{birkhoff},  ~\cite[Section 8.1]{coddington} that

\begin{equation}
\label{ineq-intermedzone:3}
z_{i+1} -z_{i}  \leq 2\pi\mu_k^{-1}\sqrt{ \frac{c_M}{\eps}}, \quad 0\leq i\leq \mathfrak{s} - 1.
\end{equation}

  The following claim extends ~\eqref{ineq-intermedzone:3} and will be useful in the sequel. It says that the distance between two consecutive zeros of an (oscillatory) eigenfunction can be made arbitrarily small, if we drop a finite number of eigenfunctions associated with ``low'' eigenvalues. The threshold $\lambda_{\alpha,\varepsilon}$ applies uniformly to all coefficients $c(y)\in \mathscr{K}.$

\begin{claim}
\label{claim:1}
Let $c(y)\in\mathscr{K}.$ For each $\alpha>0,\,\eps>0$ there exists $\lambda_{\alpha,\varepsilon}$ such that $\lambda>\lambda_{\alpha,\varepsilon}$
 implies
 $$z_{i+1} -z_{i} < \alpha,\quad 0\leq i\leq \mathfrak{s} - 1.$$

 In particular, $\lambda_{\alpha,\varepsilon}$ can be chosen uniformly for all $c(y)\in\mathscr{K}.$ For each $c(y)\in\mathscr{K}$ there are at most finitely many eigenvalues $\lambda<\lambda_{\alpha,\varepsilon}.$
\end{claim}
\begin{proof}
Recall that we are assuming $(\mu_k,\lambda) \in\mathscr{A}^{c}_\eps$ so that by ~\eqref{ineq-intermedzone:2}
$p(\ylambk)\geq \frac{\eps}{c_M}.$ Thus by the comparison principle it suffices to compare ~\eqref{eqnonguide} with the constant coefficient equation $v''+\gamma^2v=0$ for any large $\gamma.$ Pick $\mu_{0,\varepsilon}>\gamma\sqrt{\frac{c_M}{\eps}}.$  Then  $p(\ylambk)\mu_k^2>\gamma^2$ if $\mu_k>\mu_{0,\varepsilon}.$
 Next choose $\lambda_{0,\varepsilon}$ such that $\frac{\lambda_{0,\varepsilon}}{c_M}-\mu_{0,\varepsilon}^2>\gamma^2.$ Clearly for any $\lambda>\lambda_{0,\varepsilon}$ and $\mu_k\leq \mu_{0,\varepsilon}$ we have $\mu_k^2p(\ylambk)>\gamma^2.$
 Note that there are at most finitely many pairs $(\mu_k,\lambda)\in \mathscr{A}^{c}_\eps$ with $\mu_k\leq\mu_{0,\varepsilon}$ and $\lambda\leq \lambda_{0,\varepsilon}.$
   Finally, take $\lambda_{\alpha,\varepsilon}=\max \Big[(c_M+\eps)\mu_{0,\varepsilon}^2,\lambda_{0,\varepsilon}\Big].$
   \end{proof}

 In what follows  we assume that $c(y)\in \mathscr{K}$ and consider spectral values $(\mu_k,\lambda)\in\mathscr{A}^{c}_\eps.$ The following claim, an immediate consequence of \eqref{ineq-intermedzone:2},  will be useful in the sequel, when estimating masses of eigenfunctions in intervals.
\begin{claim}
\label{claim:2} Let $u(\ylambk)$  be a solution to ~\eqref{eqnonguide}, where
$(\mu_k,\lambda) \in\mathscr{A}^{c}_\eps.$
Then it is strictly convex (or concave) in every interval
\begin{equation*}
y\in (z_{i},z_{i+1}),\quad 0\leq i\leq \mathfrak{s}.
\label{claim-ng:1}
\end{equation*}
\end{claim}
\noindent
In particular, without further assumptions, the solutions are oscillatory and are  convex (or concave) between consecutive zeros. However, in various sub-intervals their amplitudes might decay to zero, hence concentrating in the complementary domain. It is precisely this behavior that we seek to exclude.

We start off with the classical case of a $C^2$ coefficient $c(y).$
In this case, a full asymptotic characterization of the eigenfunctions is possible.
\subsection{\textbf{THE REGULAR CASE: THE LIOUVILLE TRANSFORMATION WITH $c(y)\in C^{2}(\lbrack 0,H\rbrack)$}}
\label{section-Liouville-ng:1}
$\empty$\\
This case is of interest, as it yields an almost sinusoidal behavior of the eigenfunctions, not only estimates on the mass in a band.

\begin{thm}
\label{thm-ng-Liouville:1-new} Let $\Lambda>\eps$ and set
$$\mathscr{A}^{c}_{\eps,\Lambda}= \mathscr{A}^{c}_\eps\cap \set{\lambda\leq(c_{M} + \Lambda)\mu_{k}^{2}}.$$

Assume  that $c(y)\in C^{2}(\lbrack 0,H\rbrack).$ Let $u(\ylambk)$ be a normalized solution to \eqref{eqnonguide}. Then for every interval $(a,b)\subseteq (0,H)$
\begin{equation}
\label{ineq-ng-Liouville:1-new}
\inf_{(\mu_{k},\lambda)\in \mathscr{A}^{c}_{\eps,\Lambda}}\int_{a}^{b} u(\ylambk)^{2}c(y)^{-1}dy>0,
\end{equation}
that is equivalent to $\inf_{(\mu_{k},\lambda)\in \mathscr{A}^{c}_{\eps,\Lambda}}\int_{\Omega_{a,b}}v_{\lambda}(x)^{2}\tilde{c}(x)^{-1}dx>0.$
\end{thm}

\begin{proof} The hypotheses imposed in the theorem imply that
\begin{equation*}
\label{ineq-intermedzone:1}
\frac{c_{M}}{c(y)} - 1 +\frac{\varepsilon}{c(y)}\leq p(\ylambk))\leq \frac{c_{M}}{c(y)} - 1 +\frac{\Lambda}{c(y)},\quad y\in \lbrack 0, H\rbrack.
\end{equation*}

 Thus there exist constants $0<\kappa_{1}=\frac{\eps}{c_M}<\kappa_{2}=\frac{c_M+\Lambda}{c_m}$  such that
\begin{equation}
\label{ineq-ing-Liouville:2}
\kappa_{1} \leq p(\ylambk)<\kappa_{2}, \quad y\in \lbrack 0, H\rbrack,\,\,(\mu_k,\lambda)\in\mathscr{A}^{c}_{\eps,\Lambda}.
\end{equation}

 We apply the Liouville transformation ~\cite[Chapter IV]{erdelyi}:

       \be\label{eqliouv}
       \xi=\int_0^y\sqrt{p(t;\lambda,k)}dt,\quad \eta(\xi;\lambda,k)=[p(y;\lambda,k)]^{\frac14}u(\ylambk).
       \ee
     Note that $\xi\in[0,\overline{H}],\,\,\overline{H}=\int_0^H \sqrt{\frac{\lambda}{\mu^2_{k}c(t)}-1}\,\,dt\leq H\sqrt{\frac{c_M+\Lambda}{c_m}-1}.$

     The function $\eta(\xi;\lambda,k)$ satisfies the equation
     \be\label{eqetaxi}
     \frac{d^2\eta}{d\xi^2}+\mu_k^2\eta=\rho(\xi;\lambda,k)\eta,
     \ee
       where
       $$\rho(\xi;\lambda,k)=\frac14\frac{p''(\ylambk)}{p(\ylambk)^2}-
       \frac{5}{16}\frac{p'(\ylambk)^2}{p(\ylambk)^3}.$$

    Note that the form of ~\eqref{eqetaxi} is the starting point for the asymptotic behavior of solutions involving potential perturbations. However in our case  the potential depends on the spectral parameter.

        Observe that under our assumptions
         the family
          $$\mathscr{B}=\set{\rho(\xi;\lambda,k),\quad (\mu_k,\lambda)\in\mathscr{A}^{c}_{\eps,\Lambda}}$$
           is \textbf{uniformly bounded}.

       Since $\eta(0;\lambda,k)=0,$ Equation ~\eqref{eqetaxi}
       entails, with $\alpha=\mu_k^{-1}\eta'(0;\lambda,k),$
       \be\label{eqetaxiinteg}
       \eta(\xi;\lambda,k)=\alpha \sin(\mu_k\xi)+
       \mu_k^{-1}\int_0^\xi\sin(\mu_k(\xi-\tau))
       \rho(\tau;\lambda,k)
       \eta(\tau;\lambda,k)d\tau,\quad
       \xi\in[0,\overline{H}].
       \ee

       The uniform boundedness of the family $\mathscr{B}$ implies that
              the Volterra integral equation
       ~\eqref{eqetaxiinteg} is
       solvable for any sufficiently large $\mu_k$ and furthermore, for any small $\delta>0$, there exists $\mu_{0}>0$ so that
         \be\label{eqcompareetasin}
         |\eta(\xi;\lambda,k)-\alpha \sin(\mu_k\xi)|\leq \delta,\quad \xi\in[0,\overline{H}],\,\,\mu_k>\mu_0.
         \ee
         We now make the following observations.
         \begin{itemize}
         \item Recall that $\int_0^H u(\ylambk)^2c(y)^{-1}dy=1,$ hence in light of ~\eqref{ineq-ing-Liouville:2} and $\lambda\leq (c_{M} + \Lambda)\mu_{k}^{2}$
          there exist two constants $0<\zeta_1<\zeta_2<\infty$ so that
             \be\label{eqeta2integ}
            \zeta_1\leq \int_0^{\overline{H}}\eta(\xi;\lambda,k)^2d\xi\leq \zeta_2,\quad \mu_{k} >\mu_0.
             \ee
             \item It follows from ~\eqref{eqcompareetasin}-\eqref{eqeta2integ} that there exist two constants $0<r_1<r_2<\infty$ so that
             \be\label{eqalphar1r2}
                r_1\leq |\alpha|\leq r_2, \quad \mu_{k}>\mu_0.
             \ee
         \end{itemize}
             We conclude (again from ~\eqref{eqcompareetasin}) that for every interval $(\xi_1,\xi_2)\subseteq [0,\overline{H}]$
             $$ \int_{\xi_1}^{\xi_2}\eta(\xi;\lambda,k)^2d\xi\geq \alpha^2\frac{\xi_2-\xi_1}{4}-\delta (\xi_2-\xi_1),\quad \mu_{k}>\mu_0.$$
             Note that there are at most finitely many normalized eigenfunctions associated
              with values $\mu_k<\mu_0,$ since $\lambda$ is bounded from above.

              Switching back to the original variable $y$ and the function $u(\ylambk)$  we get ~\eqref{ineq-ng-Liouville:1-new}.
              \end{proof}
  \begin{rem}
  \label{rem-ng:1} Note that the hypotheses of Theorem \ref{thm-ng-Liouville:1-new} entail not only the conclusion that the eigenfunctions do not concentrate in sub-domains of $(0,H)$ but also their asymptotic (sinusoidal) form, as in \eqref{eqcompareetasin}.
  \end{rem}
  \subsection{\textbf{BEYOND THE REGULAR CASE (SEE THEOREM ~\ref{theo-C1-noguided})}}
\label{subsection-sufficient condition-ng}
$\empty$\\
 The implications of the assumption that $c(y)$ is subject only to the {\textbf{\textit{minimal amplitude hypothesis}}}  ( Definition ~\ref{definition-hypo-minam:1}) will now be studied. No regularity is required of $c(y),$
and only condition $\mbox{\bf(H)}$ (see ~\eqref{eqassumptioncy}) is imposed.\\
We have already seen that the lack of regularity does not affect the oscillatory character of the solutions. The remaining issue is to see that the masses of the oscillatory solutions  $\set{ u(\ylambk),\,(\mu_{k},\lambda)\in\mathscr{A}^{c}_\eps}$  in any interval remain uniformly bounded away from zero.\\ This is addressed in the following theorem
which is a somewhat more detailed form of  Theorem  ~\ref{theo-C1-noguided}. Its proof is straightforward, reducing the non-concentration issue to a study of the minimal amplitude hypothesis for various functional classes.

\begin{thm}
 \label{propcyirreg} Let $c(y)\in \mathscr{K}$ satisfy the minimal amplitude hypothesis (Definition ~\ref{definition-hypo-minam:1}).
 Consider the family $\set{ u(\ylambk),\,(\mu_{k},\lambda)\in\mathscr{A}^{c}_\eps}$ of  normalized solutions to \eqref{eqnonguide}.

                Then, for every interval $(a,b)\subseteq (0,H),$ there exist  constants
                \begin{itemize}
                \item $d>0$ depending on $\eps,\,c_M,\,c_m,\, b-a,\,\mathfrak{r_c},$
                    \item $\lambda_0>0$ depending on $\eps,\,c_M,\,c_m,\, b-a$
                \end{itemize}
                              such that
                \be\label{equyneq0}
                \int_{a}^{b}\, u(y;\lambda,k)^2c(y)^{-1}dy\geq d,\quad \,\,
                 \lambda>\lambda_0.
                \ee
                 This estimate is equivalent to  $\int_{\Omega_{a,b}}\,v_{\lambda}(x)^2\tilde{c}(x)^{-1}dx\geq d,$   where $v_\lambda$ is the eigenfunction of $-\tilde{c}\Delta$ associated with $\lambda.$
             \end{thm}

             \begin{proof}

 By Claim ~\ref{claim:2}  the graph of $u(\ylambk)$ is convex (or concave) in $(z_{i},z_{i+1})$ and there exists a unique point $z_{i+\frac{1}{2}}\in (z_{i},z_{i+1})$  so that
 \begin{equation}
\label{eq-intermedzone:2}
u(z_{i+\frac{1}{2}};\lambda,k)^{2}\mbox{ is maximal in the interval and } u'(z_{i+\frac{1}{2}};\lambda,k) =0.
 \end{equation}
In particular,
 \begin{equation}
\label{ineq-intermedzone:5}
u(z_{i+\frac{1}{2}};\lambda,k)^{2}\geq \mathfrak{r}_c^{2}, \quad 0\leq i \leq \mathfrak{s}.
\end{equation}
Let $\bar{u}(\ylambk)$ be the continuous piecewise linear function connecting $(z_{i},0)$ to $( z_{i+\frac{1}{2}}, u(z_{i+\frac{1}{2}};\lambda,k))$ then to $(z_{i+1},0).$ It is readily verified  that
 \begin{equation}
\label{eq-intermedzone:3}
 \int_{z_{i}}^{z_{i+1}}\bar{u}(\ylambk)^{2}dy = \frac{1}{3}(z_{i+1}-z_{i})u(z_{i+\frac{1}{2}};\lambda,k)^{2},
 \end{equation}
and, due to the convexity,
  \begin{equation*}
\label{ineq-intermedzone:6}
 \int_{z_{i}}^{z_{i+1}}u(\ylambk)^{2}dy \geq  \int_{z_{i}}^{z_{i+1}}\bar{u}(\ylambk)^{2}dy .
\end{equation*}
By ~\eqref{eq-intermedzone:3} and ~\eqref{ineq-intermedzone:5}
 \begin{equation}
\label{ineq-intermedzone:7}
 \int_{z_{i}}^{z_{i+1}}u(\ylambk)^{2}dy \geq \frac{1}{3}\mathfrak{r}_c^{2}(z_{i+1}-z_{i}).
\end{equation}

Let $(a,b)\subseteq \lbrack 0, H\rbrack.$ In view of Claim ~\ref{claim:1} we have $\lambda_0=\lambda_0(b-a,\eps)>0$ such  that for any $\lambda>\lambda_0$ there are zeros $a\leq z_{i_0}<z_{i_{1}}\leq b$ and $|a-z_{i_0}|+|b-z_{i_{1}}|<\frac13(b-a).$ In particular, summing in ~\eqref{ineq-intermedzone:7} yields
\begin{equation}
\label{ineq-intermedzone:8-new}
 \int_{a}^{b}u(\ylambk)^{2}dy \geq \int_{z_{i_0}}^{z_{i_1}}u(\ylambk)^{2}dy \geq\frac13\cdot\frac23\mathfrak{r}_c^{2}(b-a),
\end{equation}

 which concludes the proof of the theorem. \end{proof}

\subsection{\textbf{MORE ON THE MINIMAL AMPLITUDE HYPOTHESIS}}$\empty$\\

Recall that the minimal amplitude was defined by ~\eqref{equation-definition-hypo-minam:1}:

$$ \mathfrak{r}_c^{2} = \inf_{\begin{array}{c} y\in \lbrack 0, H\rbrack,\\ (\mu_k,\lambda)\in\mathscr{A}^{c}_\eps\end{array}}\lbrack u_{\lambda,k}(y)^{2} + u'_{\lambda,k}(y)^{2}\rbrack.$$
We now consider this quantity in more detail.

The next proposition supplements Equation \eqref{ineq-intermedzone:5}.
\begin{prop}
\label{prop-minam:0}
  Let $u(y;\lambda,k)$ a normalized solution to \eqref{eqnonguide} where $c(y)\in \mathscr{K}$ and $(\mu_{k},\lambda)\in\mathscr{A}^{c}_\eps.$ Let
\begin{equation}
\label{definition-minimal-amplitude:2}
\mathfrak{r}_{c,\mu_{k},\lambda}^{2}= \inf_{y\in \lbrack 0, H\rbrack}\lbrack u(y;\lambda,k)^{2} + u'(y;\lambda,k)^{2}\rbrack.
\end{equation}
Then there exists a positive
$\tilde{\lambda}_0$ such that for any  $\lambda>\tilde{\lambda}_0,$
\begin{equation}
\label{definition-minimal-amplitude:3}
\mathfrak{r}_{c,\mu_{k},\lambda}^{2} =\min_{0\leq i < \mathfrak{s}} u(z_{i + \frac{1}{2}};\lambda,k)^{2}.
\end{equation}
\end{prop}

\begin{proof} Suppose that $u(y;\lambda,k)$ is convex in the interval $(z_{i}, z_{i+1})$ (see Claim ~\ref{claim:2}). In light of Equation  ~\eqref{eqnonguide} the function $u(y;\lambda,k)$ satisfies
\begin{equation}
\label{equation-minimal-amplitude:1}
u''(y;\lambda,k) +\mu^{2}_{k}p(y;\lambda,k) u(y;\lambda,k) = 0,
\end{equation}
hence
\begin{equation}
\label{equation-minimal-amplitude:2}
\frac{d}{dy}\lbrack u(y;\lambda,k)^{2} + u'(y;\lambda,k)^{2}\rbrack = 2 \big(1-\mu^{2}_{k}p(y;\lambda,k)\big) u'(y;\lambda,k)u(y;\lambda,k).
\end{equation}
In the interval we have $u(y;\lambda,k)<0$ and by convexity
\begin{equation*}
u'(y;\lambda,k)\left\lbrace\begin{array}{l}
<0, y\in \lbrack z_{i}, z_{i+\frac{1}{2}}),\\
>0, y\in ( z_{i+\frac{1}{2}}, z_{i+1}\rbrack.
\end{array}\right.
\end{equation*}
               As in the proof of  Claim ~\ref{claim:1}, we can now find  $\tilde{\lambda}_0$ so that $\mu^{2}_{k}p(y;\lambda,k)>1$ for $\lambda> \tilde{\lambda}_0.$
               Inserting this in ~\eqref{equation-minimal-amplitude:2} we obtain

\begin{equation*}
\frac{d}{dy}\lbrack u(y;\lambda,k)^{2} + u'(y;\lambda,k)^{2}\rbrack\left\lbrace\begin{array}{l}
<0, y\in ( z_{i}, z_{i+\frac{1}{2}}),\\
>0, y\in ( z_{i+\frac{1}{2}}, z_{i+1}).
\end{array}\right.
\end{equation*}
Equation \eqref{definition-minimal-amplitude:3} clearly follows by considering all intervals.
\end{proof}

\begin{cor}
\label{cor-ng-minam-hyp:1}
If  $\lambda>\tilde{\lambda}_0$ then
for every $0\leq i < \mathfrak{s}$
\begin{equation}
\min \lbrack \vert u'(z_{i};\lambda,k)\vert^{2}, \vert u'(z_{i+1};\lambda,k)\vert^{2}\rbrack\geq \vert u(z_{i+\frac{1}{2}};\lambda,k)\vert^{2}.
\end{equation}
\end{cor}



\subsection{\textbf{DIFFUSION COEFFICIENTS SATISFYING THE MINIMAL AMPLITUDE HYPOTHESIS}}$\empty$\\

 In this subsection we present several  subsets of the set $\mathscr{K}$ of diffusion coefficients $c(y)$ for which the minimal amplitude hypothesis can be verified.  In fact, our ultimate subset is that of functions of \textit{bounded total variation} (see Subsection ~\ref{subsecbTV}), that contains all the subsets considered here. To justify our special treatment of the more restricted subsets, we call attention to the following points.
  \begin{itemize}
  \item As we narrow down the admissible coefficients $c(y)$ (as we did above for $c(y)\in C^2$) we can extract more information on the general structure of the corresponding non-guided eigenfunctions.
      \item The methods of proof in the different cases are quite different from each other. Given the important role of the minimal amplitude in these investigations and the fact that the most general case (namely \textit{all} $c(y)\in\mathscr{K}$) \textit{is still open}, it seems to us worthwhile to expound the various methods.
      \item The special case of piecewise constant coefficients is in the focus of much of the physical literature, and some of the estimates obtained in this context will prove to be crucial in establishing the more general case.

  \end{itemize}

\subsubsection{\textbf{First case}}$\empty$\label{subsublip}

\begin{framed}\textbf{$c(y)$ LIPSCHITZ}\end{framed}

   The first subset to be considered in the following proposition is that of Lipschitz functions.

\begin{prop}
\label{prop-minam:1bis}
Let $c(y)\in \mathscr{K}$ be in $W^{1,\infty}(\lbrack 0, H\rbrack).$ Then it satisfies the minimal amplitude hypothesis  with respect to $\mathscr{A}^{c}_\eps.$
\end{prop}

\begin{proof}
Let $u(\ylambk)$ be a normalized solution to ~\eqref{eqnonguide}.
 We just need to prove the estimate ~\eqref{eqrcgreat0}. Equation \eqref{eqnonguide} can be rewritten as
 \begin{equation}
\label{eq-intermedzone:4}
u''(\ylambk) + \alpha^{2}(\ylambk)u(\ylambk) = 0
\end{equation}
where
 \begin{equation}
\label{eq-intermedzone:5}
\alpha(\ylambk) = \mu_{k}\sqrt{p(\ylambk)}.
\end{equation}
Observe that in light of ~\eqref{ineq-intermedzone:2}
 \begin{equation}
\label{ineq-intermedzone:10}
\sqrt{\frac{\eps}{c_M}}\mu_{k} \leq\alpha(\ylambk).
\end{equation}
\noindent
We now replace $u,\,u'$ by $u_1,\,u_2$ as follows (suggested in the  recent paper ~\cite{ABM:1}).
\begin{equation*}
u_{1}(\ylambk) = u(\ylambk), \quad u_{2}(\ylambk) = \frac{u'(\ylambk)}{\alpha(\ylambk)},
\end{equation*}
and note that the vector function $U(\ylambk)  = \left(\begin{array}{c}
u_{1}(\ylambk) \\
u_{2}(\ylambk)
\end{array}\right)$
 satisfies
\begin{equation}
\label{eq-intermedzone:6}
U'(\ylambk) = \left(\begin{array}{cc}
0 & \alpha(\ylambk)\\
-\alpha(\ylambk) & 0
\end{array}\right)U(\ylambk)  + \left(\begin{array}{cc}0&0\\0&-\frac {\alpha'(\ylambk)}{\alpha (\ylambk)}
 \end{array}\right)U(\ylambk).
\end{equation}
It readily follows that
\begin{equation}
\label{eq-intermedzone:6a}
\frac{d}{dy}(u_{1}^{2} + u_{2}^{2}) = - 2\frac {\alpha'(\ylambk)}{\alpha (\ylambk)} u_{2}^{2}.
\end{equation}

  Now
  $$\frac{\alpha'}{\alpha}=\frac{1}{2p}p',$$
  hence
  $$\frac{\alpha'}{\alpha}=
  -\frac12\,\frac{\mu_k^2c(y)}{\lambda-\mu_k^2c(y)}\,\frac{\lambda c'(y)}{\mu_k^2c(y)^2}
  =-\frac12\,\frac{\lambda}{\lambda-\mu_k^2c(y)}\,\frac{c'(y)}{c(y)}.$$
  Since $\mu_k^2c(y)\leq \mu_k^2c_M=\mu_k^2c_M\frac{c_M+\eps}{c_M+\eps}\leq\frac{c_M}{c_M+\eps}\lambda$ it follows that
  $$\frac{\lambda}{\lambda-\mu_k^2c(y)}\leq\frac{1}{1-\frac{c_M}{c_M+\eps}}=\frac{c_M+\eps}{\eps}.$$
  Finally
  \begin{equation*}\label{eqalpha'alpha}
  \Big|\frac{\alpha'}{\alpha}\Big|\leq \frac{c_M+\eps}{2\eps}\,\sup\limits_{y\in [0,H]}\,\frac{|c'(y)|}{c(y)},
  \end{equation*}
  that implies
  \begin{equation}
\label{ineq-intermedzone:11}
-C (u_{1}^{2} + u_{2}^{2})\leq \frac{d}{dy}(u_{1}^{2} + u_{2}^{2})\leq C (u_{1}^{2} + u_{2}^{2}),
\quad C=\frac{c_M+\eps}{\eps}\,\sup\limits_{y\in [0,H]}\,\frac{|c'(y)|}{c(y)}.
\end{equation}

 Since $u_{1}^{2} + u_{2}^{2}>0$ in the interval $\lbrack 0, H\rbrack$ we conclude that there exists a constant $R>0,$ so that for all $(\mu_k,\lambda)\in \mathscr{A}^{c}_\eps$ and  for any $y_{1},y_{2}\in \lbrack 0, H\rbrack$
  \begin{equation}
\label{ineq-intermedzone:12}
(u_{1}^{2} + u_{2}^{2})(y_{2})\leq R (u_{1}^{2} + u_{2}^{2})(y_{1}).
\end{equation}
Furthermore, since $u(\ylambk)$ is normalized $\int_{0}^{H}u(\ylambk)^{2}c(y)^{-1}dy = 1,$ it follows that there exists a constant $\eta>0,$ so that for all $(\mu_k,\lambda)\in \mathscr{A}^{c}_\eps,$

  \begin{equation}
\label{ineq-intermedzone:13}
(u_{1}^{2} + u_{2}^{2})(y)\geq \eta, \quad y\in \lbrack 0, H\rbrack.
\end{equation}
This estimate, combined with the definition of $u_{1}, u_{2}$ and \eqref{ineq-intermedzone:10} implies the required estimate \eqref{eqrcgreat0} and concludes the proof of the proposition.
\end{proof}
\noindent
\noindent
\subsubsection{\textbf{Second case}}$\empty$\label{subsubmonotone}

\begin{framed}\textbf{$c(y)$ MONOTONE NONDECREASING}\end{framed}

In the following proposition we relax the regularity of the coefficient $c(y).$ In fact it is no more required to be continuous but on the other hand a monotonicity assumption is imposed.

\begin{prop}
\label{prop-ng-minam-mnam-coefdiff-monotone:1}
   Let $c(y)\in\mathscr{K}$  and assume  that $c(y)$ is nondecreasing.
              Then it satisfies the minimal amplitude hypothesis  with respect to $\mathscr{A}^{c}_\eps.$
\end{prop}

\begin{proof}
Let $u(\ylambk)$ be a normalized solution to ~\eqref{eqnonguide} where $(\mu_k,\lambda)\in\mathscr{A}^{c}_\eps.$  Define $\mathfrak{r}_{c,\mu_{k},\lambda}$ as in ~\eqref{definition-minimal-amplitude:2} and use the transformation
 (as in the proof of Proposition \ref{prop-minam:1bis})
\begin{equation*}
\label{equation-ng-minam-coefdiff-monotone:1}
\begin{array}{ccc}
u_{1}(\ylambk) = u(\ylambk), & &u_{2}(\ylambk) = \frac{u'(\ylambk)}{\alpha(\ylambk)},\\
 p(\ylambk)= \frac{\lambda}{\mu^2_{k}c(y)}-1, & &\alpha(\ylambk) = \mu_{k}\sqrt{p(\ylambk)}.
\end{array}
\end{equation*}
 We first show that even though  $\alpha(\ylambk)$ is not necessarily continuous, Equations ~\eqref{eq-intermedzone:6}-\eqref{eq-intermedzone:6a} can be extended (in distribution sense) so that
          \be\label{equ1squ2sq-noncon}
            \frac{d}{dy}(u_1^2+u_2^2)=-\frac{(\alpha^2(\ylambk))'}{\alpha^2(\ylambk)}u_2^2.
          \ee
   Indeed,  we first have
           $$\frac{d}{dy}u_1^2=2u_1u_1'=2\alpha u_1u_2.$$
           Now let $\set{\alpha_m(\ylambk)}_{m=1}^\infty\subseteq C^\infty[0,H]$ be a uniformly bounded sequence converging to $\alpha(\ylambk)$ a.e. (hence in distribution sense)\footnote{Choose $0\leq\phi\in C^{\infty}_{0}(\mathbb{R}), \int\phi(y)dy=1$ and a continuation $\tilde{\alpha}(y)$ of $\alpha,$ nonincreasing on $(-1,H+1),$ with compact support on $\mathbb{R}$. We set $\alpha_{m}:= (\tilde{\alpha}\ast\phi_{m})_{\vert \lbrack 0,H\rbrack}$ with $\phi_{m}(y) = \frac{1}{m}\phi (\frac{y}{m})$. From $\alpha_{m}(y)= \int \tilde{\alpha}(y-z)\phi_{m}(z)dz$, we see that $\alpha_{m}$ is nonincreasing and $\vert\alpha(y) -\alpha_{m}(y)\vert\leq \sup_{\vert z\vert<\frac{1}{m}}\vert\alpha(y) -\alpha_{m}(y-z)\vert$ that proves the convergence a.e. on $\lbrack 0, H\rbrack$, hence in $L^{1}(0,H)$}.
           We may also assume that it is uniformly bounded away from zero. Under these conditions
          the sequence  $\frac{u'(\ylambk)^2}{\alpha_m^2(\ylambk)}$ (resp.
          $\frac{d}{dy}\frac{u'(\ylambk)^2}{\alpha_m^2(\ylambk)}$) converges (in distribution sense) to $u_2(\ylambk)^2$ (resp.$\frac{d}{dy}u_2(\ylambk)^2$). Since $u'\in H^1$ we have (using
           ~\eqref{eqnonguide})
           \begin{equation*}
           \frac{d}{dy}\Big[\frac{u'(\ylambk)^2}{\alpha_m^2(\ylambk)}\Big]
           =
           \frac{-2\alpha^2(\ylambk)\alpha_m^2(\ylambk)u'(\ylambk)u(\ylambk)
           -(\alpha_m^2(\ylambk))'u'(\ylambk)^2}{\alpha_m^4(\ylambk)}.
    \end{equation*}
           Clearly, the right-hand side in this equation converges, in the sense of distributions, to
             $$-2u'(\ylambk)u(\ylambk)-\frac{(\alpha^2(\ylambk))'u'(\ylambk)^2}{\alpha^4(\ylambk)}$$
             and substituting $u'(\ylambk)=\alpha(\ylambk)u_2(\ylambk)$ we obtain ~\eqref{equ1squ2sq-noncon}.
        Since $(\alpha^2(\ylambk))'\leq 0$  we conclude that
        $$ $$
\centerline{\bf $(\ast)\qquad$ The function $u_{1}^{2}(\ylambk) + u_{2}^{2}(\ylambk)$ is nondecreasing in $y\in \lbrack 0, H\rbrack.$}

$$ $$
As above, let
\begin{equation*}
\label{definition-ng-minam-coefdiff-monotone:1}
Z(\lambda,k) = \{z_{0} =0<z_{1}<z_{2}<\ldots,z_{ \mathfrak{s}} = H, u(z_{i};\lambda, k) = 0, \quad 0\leq  i\leq \mathfrak{s}\}\subseteq \lbrack 0, H\rbrack
\end{equation*}
be the set of zeros of  $u(\ylambk).$  Recall that the set $\set{
z_{i+\frac{1}{2}}}$ is defined as in ~\eqref{eq-intermedzone:2}.\\
 From Proposition ~\ref{prop-minam:0} and $(\ast)$ we deduce that there is $\tilde{\lambda}_0$ such that for any  $\lambda>\tilde{\lambda}_0,$
\begin{equation}
\label{equation-ng-minam-coefdiff-monotone:3}
u(z_{i+\frac{1}{2}};\ylambk)^{2}\leq u(z_{i+\frac{3}{2}};\ylambk)^{2}, \quad 0\leq i<\mathfrak{s}-1
\end{equation}
and
\begin{equation}
\label{equation-ng-minam-coefdiff-monotone:4}
\mathfrak{r}^{2}_{c,\mu_k, \lambda} = u(z_{\frac{1}{2}};\ylambk)^{2}.
\end{equation}
The proof will be complete if we prove
\begin{equation}
\label{equation-ng-minam-coefdiff-monotone:11}
 \tilde{\mathfrak{r_c}}=\inf_{\begin{array}{c} \lambda>\tilde{\lambda}_0,\\ (\mu_k,\lambda)\in\mathscr{A}^{c}_\eps\end{array}} \mathfrak{r}_{c,\mu_k,\lambda}>0.
\end{equation}
Indeed, there are only finitely many eigenfunctions with $(\mu_k,\lambda)\in\mathscr{A}^{c}_\eps$ and $\lambda\leq \tilde{\lambda}_0,$
  hence $$\mathfrak{r_c} =\min\set{\tilde{\mathfrak{r_c}},\min\limits_{ \lambda\leq\tilde{\lambda}_0,\, (\mu_k,\lambda)\in\mathscr{A}^{c}_\eps} \mathfrak{r}_{c,\mu_k,\lambda}}.$$

Note that  by excluding at most a finite number of eigenvalues we shall be able to obtain a more explicit lower bound for
the $\tilde{\mathfrak{r_c}}$ in ~\eqref{equation-ng-minam-coefdiff-monotone:11}. This will be evident in Corollary ~\ref{corunifcmonot}  below.

 It remains to prove  ~\eqref{equation-ng-minam-coefdiff-monotone:11}. This is done in two steps.
 \begin{itemize}
 \item \textbf{Controlling the nonincreasing function $\alpha(y;\lambda,k)^{2}.$}\\

For $i<j$ we have $\alpha(z_{j+\frac{1}{2}};\lambda,k)^{2} \leq \alpha(z_{i+\frac{1}{2}};\lambda,k)^{2}. $ On the other hand, by the monotonicity of $c(y)$ and
~\eqref{eq-intermedzone:5}-~\eqref{ineq-intermedzone:10} there exists a constant $\gamma>0,$ depending only on $\eps,\,c_m,\,c_M$ such that for $i<j$

\be\label{eqalphaialphaj}
1\leq\frac{\alpha(z_{i+\frac{1}{2}};\lambda,k)^{2}}
{\alpha(z_{j+\frac{1}{2}};\lambda,k)^{2}}\leq \gamma,\quad (\mu_k,\lambda)\in \mathscr{A}^{c}_\eps.
\ee

\item \textbf{Controlling the extremal values $u(z_{i+\frac{1}{2}};\lambda,k)^{2}$.}
\\
Integrate the equation
\begin{equation*}
\label{equation-ng-minam-coefdiff-monotone:6}
\frac{d}{dy}u'(y;\lambda ,k)^{2} = -\alpha(y;\lambda,k)^{2}\frac{d}{dy}u(y;\lambda ,k)^{2}, \quad \alpha(y;\lambda,k)^{2}= \mu_{k}^{2}p(y,\lambda,k),
\end{equation*}
 over the interval $\lbrack  z_{i+\frac{1}{2}},z_{j+\frac{1}{2}}\rbrack, i<j.$ Using Riemann-Stieljes integration by parts\footnote{Note that $u^{2}$ is continuous and $\alpha$ is a BV function. See \cite{Gordon:1}, Theorems 12.14 and 12.15.} we get
\begin{equation}
\label{equation-ng-minam-coefdiff-monotone:7}
\alpha(z_{j+\frac{1}{2}};\lambda,k) u(z_{j+\frac{1}{2}};\lambda,k)^{2} - \alpha(z_{i+\frac{1}{2}};\lambda,k) u(z_{i+\frac{1}{2}};\lambda,k)^{2} = \int_{z_{i+\frac{1}{2}}}^{z_{j+\frac{1}{2}}} u(\ylambk)^{2} d(\alpha(\ylambk)).
\end{equation}
Since $\alpha(\ylambk)^{2}$ is nonincreasing, the right-hand side is nonpositive and we conclude
\begin{equation}
\label{equation-ng-minam-coefdiff-monotone:8}
\alpha(z_{j+\frac{1}{2}};\lambda,k)^{2} u(z_{j+\frac{1}{2}};\lambda,k)^{2} \leq \alpha(z_{i+\frac{1}{2}};\lambda,k)^{2} u(z_{i+\frac{1}{2}};\lambda,k)^{2}.
\end{equation}

We have $u(z_{i+\frac{1}{2}};\lambda,k)^{2}\leq u(z_{j+\frac{1}{2}};\lambda,k)^{2}$  (see ~\eqref{equation-ng-minam-coefdiff-monotone:3}). From  ~\eqref{equation-ng-minam-coefdiff-monotone:8} and ~\eqref{eqalphaialphaj}
we infer that there exists a constant $\mathfrak{c}>0$  depending solely on $\eps,\,c_m,\,c_M$  such that
\begin{equation}
\label{equation-ng-minam-coefdiff-monotone:9}
u(z_{j+\frac{1}{2}};\lambda,k)^{2}\leq \mathfrak{c} u(z_{i+\frac{1}{2}};\lambda,k)^{2}, \quad 0\leq i<j \leq\mathfrak{s}-1,  \lambda>\tilde{\lambda}_0.
\end{equation}
\end{itemize}

Since $u(\ylambk)$ is normalized we have, in view of \eqref{eq-intermedzone:2}
\begin{equation}
\label{equation-ng-minam-coefdiff-monotone:10}
\sum_{i=0}^{\mathfrak{s}-1}u(z_{i+\frac{1}{2}};\lambda,k)^{2}(z_{i+1} - z_{i})\geq \int_{0}^{H} u(\ylambk)^{2} dy\geq c_{m}.
\end{equation}
Combining this estimate with  ~\eqref{equation-ng-minam-coefdiff-monotone:4},\,~\eqref{equation-ng-minam-coefdiff-monotone:9} we  obtain
\be\label{equnifrc}
H\mathfrak{c}^2\mathfrak{r}^{2}_{c,\mu_k, \lambda}\geq c_m,
\ee
and ~\eqref{equation-ng-minam-coefdiff-monotone:11} readily follows.
\end{proof}

In the course of the proof of Proposition ~\ref{prop-ng-minam-mnam-coefdiff-monotone:1} we have actually obtained interesting (and non-trivial) facts concerning the behavior of the normalized non-guided oscillatory solutions $u(\ylambk)$ for nondecreasing coefficients $c(y).$ They are highlighted in the  corollary below.

First, we fix $\eps>0,\,\,0<c_m<c_M,$ and define $\mathscr{K},\,\mathscr{A}^{c}_\eps$ as in ~\eqref{eq-def-K},\,~\eqref{eqAceps}, respectively.  Let $\mathscr{K}_1\subseteq\mathscr{K}$ be the set of all nondecreasing diffusion coefficients.

Observe that the threshold value $\tilde{\lambda}_0$ depends only on $\eps,\,c_m,\,c_M.$ Also the constant $\mathfrak{c}$ appearing in ~\eqref{equation-ng-minam-coefdiff-monotone:9} depends solely on these parameters.

\begin{cor}\label{corunifcmonot}
\begin{itemize} \item The minimal amplitude for all solutions with $\lambda>\tilde{\lambda}_0$ is strictly positive:
\be\label{eqinfrcK1}
\mathfrak{r}_{\mathscr{K}_1}:=\, \inf\limits_{c(y)\in \mathscr{K}_1}\,\,\inf_{\lambda>\tilde{\lambda}_0} \mathfrak{r}_{c,\mu_k,\lambda}
>0.
\ee
\item The amplitudes of any solution $u(\ylambk)$ between zeros are growing as $y$ moves from 0 to $H.$ However, the ratios of the amplitudes remain universally (for all $c(y)\in \mathscr{K}_1$)  bounded for $\lambda>\tilde{\lambda}_{0}$.
    \end{itemize}
\end{cor}

\subsubsection{\textbf{Third case}}\label{subsecpiececonst}
$$ $$
\begin{framed}\textbf{$c(y)$ PIECEWISE CONSTANT}\end{framed}

 We turn next to the case that $c(y)$ is a piecewise constant function. In  Theorem ~\ref{prop-minam:1-new} below we discuss our most general case,  namely  $c(y)$ of bounded variation. To this end, a detailed treatment of the piecewise constant case is needed.

     We shall use the following notation. There exist
 $0=h_{-1}<h_0<h_1<\ldots<h_N=H,$
           and positive constants $c_0,c_1,\ldots,c_N$ so that
           \be\label{eqhjcj-a}
            c(y)=c_{j+1},\,\,\,y\in(h_j,h_{j+1}),\,\,j=-1,0,\ldots,N-1.
           \ee
         We  show that $c(y)$ satisfies the minimal amplitude hypothesis, where the relevant constants depend only \textit{on its total variation}.

         \textbf{Notational comment:} In order to keep the notational uniformity with the other sections, we retain the notation $c_m,\,c_M$ for the minimal and maximal values, respectively, of $c(y).$ Of course they coincide with some $c_j's$ but the distinction in  various estimates  (such as ~\eqref{equnormbeta1}) will be completely clear.

        Recall that $u(\ylambk)$ satisfies Equation  ~\eqref{eqnonguide} with
                  $p(\ylambk)= \frac{\lambda}{\mu^2_{k}c(y)}-1,$ so that

                            \be\label{eqhjpj-new}
            p(\ylambk)=p_{j+1}=\frac{\lambda}{\mu^2_{k}c_{j+1}}-1,\,\,\,y\in(h_j,h_{j+1}),\,\,
            j=-1,0,\ldots,N-1.
           \ee

         \begin{prop}\label{propcyconst}  Assume  that
          $c(y)$ is piecewise constant as above, and let
            $$V=\suml_{j=0}^{N-1}|c_{j+1}-c_j|$$
            be the total variation of $c(y).$

          Let $u(\ylambk)$ be a normalized solution to ~\eqref{eqnonguide}
          where $(\mu_k,\,\lambda)\in \mathscr{A}^{c}_\eps.$

          Then
          \begin{enumerate}
          \item $c(y)$ satisfies the minimal amplitude hypothesis with respect to $\mathscr{A}^{c}_\eps.$

          \item For every $(a,b)\subseteq[0,H]$ there exist  constants $d>0,\,\,$
           $\lambda_0>0,\,$ depending only on $a,\,b,\,c_m,\,c_M,\,\eps, \,V,$ such that

         \be\label{equy2consta,b}
          \int_a^bu(\ylambk)^2c(y)^{-1}dy>d,\quad \lambda>\lambda_0.
                  \ee
         \textbf{Note in particular that $d$ and $\lambda_0$ }do not depend on the size $N$ of the partition.
         \end{enumerate}

         \end{prop}
         \begin{proof}
               In light of Theorem ~\ref{propcyirreg} we need first to prove the validity of the minimal amplitude hypothesis.

          Consider an interval $I_j=(h_j,h_{j+1}).$

         The solution $u(\ylambk)$ to ~\eqref{eqnonguide} in $I_j$ is given by
         \be\label{equyinIj-a}
         u(\ylambk)=\beta_j\sin(\mu\sqrt{p_{j+1}}(y-\gamma_j)),\quad y\in I_j
         \ee
         where $\beta_j,\,\gamma_j$ are suitable constants. Recall ~\eqref{ineq-intermedzone:2} that $|p_{j+1}|\geq\theta,\,j=-1,0,\ldots ,N-1,$
           where $\theta=\frac{\eps}{c_M}.$

        In the interval $I_j$ we have
        \be\label{equyinIjbeta}
        u(\ylambk)^2+u'(\ylambk)^2\geq\beta_j^2\min(1,\mu_k^2p_{j+1})\geq\beta_j^2\min(1,\mu_k^2\theta).
        \ee
         As $\mu_{k}\geq \mu_{1}>0$, to complete the proof we need to show the existence of a constant $\delta>0,$ depending only on $\eps,\,c_m,\,c_M,\,V$ so that
        \be\label{equyinIj-b}
        |\beta_j|\geq \delta,\quad  j=-1,0,\ldots,N-1.
        \ee
         We observe the following facts concerning the coefficients $\set{\beta_j}_{j=-1}^{N-1}.$
         \begin{itemize}
         \item $$\beta_j \neq 0, \quad j=-1,0,\ldots,N-1,$$
           since otherwise $u(\ylambk)\equiv 0.$
         \item There exists a constant $\kappa>1,$ depending only on
         $c_m,\,c_M,\,\eps, \,V,$  such that
            \be\label{eqsumbeta}\frac{\beta_j^2}{\beta_{j+1}^2},\frac{\beta_{j+1}^2}{\beta_j^2}
            \leq (1 +\kappa |c_{j+2}-c_{j+1}|),\quad j=-1,0,\ldots,N-2.\ee
                     \end{itemize}
   To establish ~\eqref{eqsumbeta} we proceed as follows.

       Denote, for $j=-1,0,\ldots,N-2$

         $$A_{j}=\mu_k\sqrt{p_{j+1}}(h_{j+1}-\gamma_j),\,\,\,
         B_{j+1}=\mu_k\sqrt{p_{j+2}}(h_{j+1}-\gamma_{j+1}).$$

         The continuity of $u(\ylambk)$ and $u'(\ylambk)$ at $h_{j+1}$ implies that
         \begin{equation*}\label{eqcontumhj}\aligned
         \beta_j\sin(A_{j})=
         \beta_{j+1}\sin(B_{j+1}),\\
         \beta_j\mu_k\sqrt{p_{j+1}}\cos(A_{j})=
         \beta_{j+1}\mu_k\sqrt{p_{j+2}}\cos(B_{j+1}).
         \endaligned
         \end{equation*}
         Recall that $\beta_j\neq 0$ for all $j.$
          It follows that for $ j=-1,0,\ldots,N-2$
          \be\label{eqbetaj2}\aligned
          \beta_j^2=\beta_{j+1}^2\Big(1+\big[\frac{p_{j+2}}{p_{j+1}}-1\big]\cos^2(B_{j+1})\Big),\\
          \beta_{j+1}^2=\beta_{j}^2\Big(1+\big[\frac{p_{j+1}}{p_{j+2}}-1\big]\cos^2(A_{j})\Big).
          \endaligned\ee
          From ~\eqref{eqhjpj-new} it readily follows that
           \begin{equation*}\label{eqpj+1pj}\aligned
           \frac{p_{j+2}}{p_{j+1}}-1=\frac{\lambda}{\lambda-\mu_k^2c_{j+1}}
           \cdot\frac{c_{j+1}-c_{j+2}}{c_{j+2}},\\\frac{p_{j+1}}{p_{j+2}}-1=\frac{\lambda}
           {\lambda-\mu^2_k c_{j+2}}
           \cdot\frac{c_{j+2}-c_{j+1}}{c_{j+1}},\quad j=-1,0,1,\ldots, N-2.
           \endaligned
           \end{equation*}

          The fact that $(\mu_k,\lambda)\in \mathscr{A}^{c}_\eps$ entails
        \be\label{eqestpj}  \Big|\frac{p_{j+1}}{p_{j+2}}-1\Big| , \Big|\frac{p_{j+2}}{p_{j+1}}-1\Big|\leq
        \kappa |c_{j+1}-c_{j+2}|, \,\,j=-1,0,\ldots,N-2,
                        \ee
                        where $\kappa>0$ depends only on $c_m,\,c_M,\,\eps.$
                        In conjunction with ~\eqref{eqbetaj2} the estimate ~\eqref{eqsumbeta} is established.

          From ~\eqref{eqbetaj2} we deduce for any $q\in\set{0,1,\ldots,N-1}$ upper and lower estimates
          $$\begin{cases}\beta_q^2\leq \beta_{-1}^2\prod\limits_{j=-1}^{q-1}(1+\kappa|c_{j+1}-c_{j}|)\leq \beta_{-1}^2e^{\kappa\suml\limits_{j=-1}^{q-1}|c_{j+1}-c_{j}|},\\
          \beta_q^2\geq\beta_{-1}^2\prod\limits_{j=-1}^{q-1}(1+\kappa|c_{j+1}-c_{j}|)^{-1}\geq \beta_{-1}^2e^{-\kappa\suml\limits_{j=-1}^{q-1}|c_{j+1}-c_{j}|}.
          \end{cases}$$
          Thus, for all $q\in\set{0,1,\ldots,N-1}$, the coefficients $\beta_{-1}$ and $\beta_q$ are comparable in the sense that
          \be\label{eqestbetal}\beta_{-1}^2e^{-\kappa V}\leq\beta_q^2\leq\beta_{-1}^2e^{\kappa V}.\ee

          The normalization of $u(\ylambk)$ in conjunction with ~\eqref{equyinIj-a}, \eqref{eqestbetal} implies
          \be\label{equnormbeta1}1=\int_0^Hu(\ylambk)^2c(y)^{-1}dy\leq c_m^{-1}\suml_{j=-1}^{N-1}\beta_j^2(h_{j+1}-h_j)\leq c_m^{-1}H\beta_{-1}^2e^{\kappa V}.\ee
          Thus finally the estimate ~\eqref{equyinIj-b} follows from ~\eqref{eqestbetal} and ~\eqref{equnormbeta1}.

          The estimates ~\eqref{equyinIjbeta} and ~\eqref{equyinIj-b} imply that the minimal amplitude hypothesis is satisfied
          \be\label{eqminampconst}
       \mathfrak{r}_c^2= \inf \set{u(\ylambk)^2+u'(\ylambk)^2,\,\,y\in[0,H]\,\,\,\,
      }
       \geq\delta^2\min(1,\mu_1^2\theta)>0,
        \ee
        and $\mathfrak{r_c}$ depends only on $\eps,\,c_m,\,c_M,\,V.$

        The non-concentration estimate ~\eqref{equy2consta,b} is now a consequence of the general Theorem ~\ref{propcyirreg}.
         \end{proof}

         In analogy with the case of the subset of all nondecreasing coefficients (Corollary ~\ref{corunifcmonot}) we deduce a similar result for all piecewise constant coefficients having a uniform bound of their total variations.
\begin{cor}\label{corunifcV} Let $\mathscr{K}_{PCV}\subseteq\mathscr{K}$ be the set of all piecewise constant diffusion coefficients, with total variation less than $V.$ Then
\begin{enumerate}
\item \be\label{eqinfrcKV}
\mathfrak{r}_{PCV}=\inf\limits_{c(y)\in \mathscr{K}_{PCV}}\mathfrak{r}_c >0.
\ee
\item The ratios of the amplitudes of any two waves (between consecutive zeros) are uniformly bounded, for all coefficients in $\mathscr{K}_{PCV}.$
\end{enumerate}
\end{cor}
\begin{proof} The estimate ~\eqref{eqinfrcKV} follows from ~\eqref{eqminampconst}. The second item follows from ~\eqref{eqestbetal}.
\end{proof}

\subsection{\textbf{THE ULTIMATE CASE- $c(y)$ OF BOUNDED VARIATION}}\label{subsecbTV} $\empty$\\

   Our ultimate result concerns the case that $$c(y)\in \mathscr{K}_V=\set{c\in\mathscr{K},\,\,TV(c)\leq V}.$$
   Recall that $\mathscr{K}$ was defined in ~\eqref{eq-def-K}.

   We establish non concentration  for spectral pairs $(\mu_k,\lambda)\in\mathscr{A}^{c}_\eps$
   (see ~\eqref{eqAceps}).

   As in the cases studied above, the proof relies on the validity of the minimal amplitude hypothesis, via the fundamental  Theorem ~\ref{propcyirreg}.

 \begin{thm}
\label{prop-minam:1-new}
 Let  $c(y)$ be of bounded variation. Then it satisfies the  minimal amplitude hypothesis with respect to $\mathscr{A}^{c}_\eps,$ uniformly for all $\set{c(y)\in\mathscr{K},\,\,TV(c)\leq V}.$

   More precisely, as in  ~\eqref{equation-definition-hypo-minam:1},
  \begin{equation}
 \label{equation-definition-cybv}
 \mathfrak{r}_c^{2} = \inf_{\begin{array}{c} y\in \lbrack 0, H\rbrack,\\ (\mu_k,\lambda)\in\mathscr{A}^{c}_\eps\end{array}}\lbrack u_{\lambda,k}(y)^{2} + u'_{\lambda,k}(y)^{2}\rbrack >0,
 \end{equation}
 and
 \begin{equation}
 \label{equation-definition-cybv-unif}
\mathfrak{r}_V^{2}:=\inf\limits_{c\in\mathscr{K}_V} \mathfrak{r}_c^{2} >0.
 \end{equation}
    Furthermore, for any $(a,b)\subseteq (0,H)$ there exists a constant $\mathfrak{f}_{a,b}>0$  such that, for every $c(y)\in\mathscr{K}_V,$ and  for every normalized $u(\ylambk)$ associated with $(\mu_k,\lambda)\in \mathscr{A}^{c}_\eps,$
    \be\label{eqnonconcgeneral}
    \int_a^bu(\ylambk)^2c(y)^{-1}dy\geq \mathfrak{f}_{a,b}.
    \ee
    Finally, let $\omega=\omega'\times(a,b).$ If $\omega'\neq\Omega'$ assume that the family $\set{\phi_k(x')}_{k=1}^\infty$ of eigenfunctions of the Laplacian in $\Omega'\subseteq\RR^d$ does not concentrate in $\Omega'\setminus\omega'.$
             \\
             Then there exists $\mathfrak{f}_\omega>0$ such  the eigenfunction $v_{\lambda}(x',y)=u(\ylambk)\phi_k(x')$ satisfies
      \be\label{eqnoncontrateab-1}
      \mathfrak{f}_\omega\leq\frac{\Vert v_{\lambda}\Vert_{L^{2}(\omega)}}{\Vert v_{\lambda}\Vert_{L^{2}(\Omega)}}\leq 1
      \ee
      uniformly for all $c(y)\in\mathscr{K}_V$ and all eigenvalues in $ \mathscr{A}^{c}_\eps.$
 \end{thm}

                The proof consists of approximating $c(y)$ by a sequence of piecewise constant functions and using the results of Proposition ~\ref{propcyconst} and Corollary ~\ref{corunifcV}. The approximation procedure is based on the following result  ~\cite[pp. 12-13]{Bres:1}.
   \begin{claim}
\label{claim:4-new}
Suppose that $c(y)\in\mathscr{K}$  and is of total variation $V>0.$
 Then there exists a sequence of piecewise constant functions $\set{c^{(n)}(y)}_{n=1}^\infty,$ so that

 \begin{itemize}
 \item \begin{equation}
 \label{constantBV:1-unif}
 \lim\limits_{n\to\infty} c^{(n)}(y)=c(y),\quad \mbox{uniformly in }\,\,y\in[0,H]. \ee
 \item \begin{equation}
 \label{constantBV:1-new}
 \sup \set{ TV(c^{(n)})}_{n=1}^\infty \leq V= TV(c).\ee
 \item \begin{equation}
 \label{constantBV:2-new}
 0<c_m=ess\,inf \,\,c(y)\leq c_m^{(n)}=ess\,inf,\, c^{(n)}(y)\leq c_M^{(n)}=ess\,sup \,\, c^{(n)}(y)\leq c_M=ess\,sup \,\,c(y),\quad n=1,2,\ldots\ee
 \end{itemize}

\end{claim}
 Recall (see Introduction) that we denote
 $$\widetilde{c}(x',y)=c(y),\,\,\,\,\widetilde{c^{(n)}}(x',y)=c^{(n)}(y),\quad x=(x',y)\in \Omega=\Omega'\times[0,H],$$
 with associated operators
$$A=-\widetilde{c}\Delta,\,\,\quad A^{(n)}=-\widetilde{c^{(n)}}\Delta.$$
 For the Laplacian $-\Delta_{x'}$ acting in $L^{2}(\Omega')$ with domain $H^2(\Omega')\cap H^1_0(\Omega')$, we denote by $(\mu_k^2,\phi_k)_{k\geq 1}$ the sequence of normalized eigenfunctions and their associated eigenvalues, ordered by $\mu_k\leq \mu_{k+1}.$ \\
 The eigenfunctions of $A$ (resp. $A^{(n)}$) are
 $$v_\lambda(x)=u(y;\lambda,k)\phi_k(x'),\quad v_\lambda^{(n)}(x)=u^{(n)}(y;\lambda,k)\phi_k(x'),$$
 where $u(y;\lambda,k)$ (resp. $u^{(n)}(y;\lambda,k)$) satisfies Equation
         ~\eqref{eqnonguide} (resp. ~\eqref{eqnonguide} with $c,\,p$ replaced by $c^{(n)},\,p^{(n)}$) and is normalized in $L^{2}((0,H), c(y)^{-1}dy)$ (resp.
         $L^{2}((0,H), c^{(n)}(y)^{-1}dy)$).

         We consider eigenfunctions associated to spectral pairs  $(\mu_k,\lambda)\in\mathscr{A}^{c}_\eps.$

           The following perturbation lemma  is at the basis of the proof of the theorem. We postpone its proof to the end of this section, following the proof of the theorem. Note that in this lemma no assumption is needed concerning the total variations of the involved functions.

         \begin{lem}[ \textbf{Convergence of eigenvalues and eigenfunctions}]
\label{lem-eigenfunction:1-new}
Let $\set{c^{(n)}}_{n=1}^\infty\subseteq \mathscr{K}$ converge uniformly (in $[0,H]$) to $c(y).$ Let $A^{(n)}$ and $A$ be the corresponding operators. Let $\lambda>0$ be an eigenvalue of $A,$ with associated normalized eigenfunction $u(y;\lambda,k)\phi_k(x').$ Then there exist $N>0$ and a sequence of eigenvalues $\set{\lambda^{(n)}}_{n=N}^\infty$ of $\set{A^{(n)}}_{n=N}^\infty,$ with associated normalized eigenfunctions $\set{u^{(n)}(y;\lambda^{(n)},k)\phi_{k}(x')}_{n=N}^\infty$ such that
  \begin{equation}
  \label{eq-conv-eigen}
  \begin{array}{l}
  (i)\hspace{50pt} \,\lim_{n\to\infty}\lambda^{(n)}=\lambda,\\
   (ii)\hspace{48pt} \lim_{n\to\infty}u^{(n)}(\cdot;\lambda^{(n)},k)
   =u(\cdot;\lambda,k) \,\,\mbox{in}\,\,H^2(0,H).
   \end{array}
  \end{equation}
 \end{lem}

\begin{proof}[\textbf{Proof of Theorem ~\ref{prop-minam:1-new}}] Pick some $\lambda \geq (c_M+\eps)\mu_k^2.$

Let $\set{u^{(n)}(y;\lambda^{(n)},k)}_{n=N}^\infty$ be a sequence as in Lemma ~\ref{lem-eigenfunction:1-new}. Note that the convergence ~\eqref{eq-conv-eigen}(i) implies that, for sufficiently large index $n$ the condition $\lambda^{(n)}>(c_M+\frac{\eps}{2})\mu_{k_j}^2$ holds. In view of the uniform bound ~\eqref{constantBV:1-new} on total variations we can invoke Corollary ~\ref{corunifcV} to get
\be\label{equnifrnj}
(\mathfrak{r}_c^{app})^2:= \inf_{N\leq n<\infty}\inf_{y\in [0,H]}\Big(u^{(n)}(y;\lambda^{(n)},k)^2+u^{(n)'}(y;\lambda^{(n)},k)^2\Big)>L^2,
\ee
where $L>0$ depends only on $\eps,\,V,\,c_m,\,c_M$ (see ~\eqref{eqminampconst}).

The $H^2$ convergence ~\eqref{eq-conv-eigen}(ii) entails uniform convergence of both the functions and their derivatives. Hence
\be\label{eqinfuymuk}u(y;\lambda,k)^2+u'(y;\lambda,k)^2\geq L^2.
\ee
The estimate ~\eqref{equation-definition-cybv} now follows from the fact that, in view of Corollary ~\ref{corunifcV}, the estimate ~\eqref{equnifrnj} holds uniformly for all approximating sequences for any solution $u(y;\lambda,k)$ associated with $(\mu_k,\lambda)\in \mathscr{A}^{c}_\eps.$ In fact, we get the uniform estimate  ~\eqref{equation-definition-cybv-unif} since $\mathfrak{r}_c$ depends only on $V.$

Finally, we turn to the non-concentration statement ~\eqref{eqnonconcgeneral}. The general Theorem ~\ref{propcyirreg} ensures the existence of $d>0,\,\lambda_0>0$ depending on $\eps,\,c_M,\,c_m,\, b-a,\,\mathfrak{r_c},$          such that

              $$  \int_{a}^{b}\, u(y;\lambda,k)^2c(y)^{-1}dy\geq d,\quad
                 \lambda>\lambda_0.$$
                 Due to ~\eqref{equation-definition-cybv-unif} this estimate is uniformly valid (with the same $d,\lambda_0$) for all $c(y)\in\mathscr{K}$ with $TV(c)\leq V.$

However, for every $c(y)\in\mathscr{K}$ with $TV(c)\leq V$ there are finitely many eigenfunctions that are excluded, namely, those with $\lambda<\lambda_0.$ Clearly, these eigenfunctions vary with $c(y).$ We now show that they can be included in ~\eqref{eqnonconcgeneral}. The price to be paid is that the lower bound $\mathfrak{f}$ depends in a more delicate way on the various parameters (and not only on $\eps,\,c_M,\,c_m,\, b-a,\,\mathfrak{r_c}$).

To obtain a contradiction let $\set{c_n(y),\,\,TV(c_n)\leq V}_{n=1}^\infty\subseteq\mathscr{K}.$
Let $\set{\lambda_n}_{n=1}^\infty$ be a sequence of eigenvalues with associated normalized eigenfunctions $\set{u_n(y;\lambda_n,k_n)}_{n=1}^\infty$ satisfying
  \be\label{eqnonguide-n1}
        u_n''(y;\lambda_n,k_n)+\Big(\frac{\lambda_n}{c_n(y)}-\mu^2_{k_n}\Big)u_n(y;\lambda_n,k_n)=0.
\ee
  Assume further that
  $$(c_M+\eps)\mu^2_{k_n}<\lambda_n<\lambda_0,\quad n=1,2,\ldots$$
  Suppose that for some interval $(0,H)$ and some subsequence (we do not change indices)
  \be\label{equntozero}
  \lim\limits_{n\to\infty} \int_a^bu_n(y;\lambda_n,k_n)^2c(y)^{-1}dy=0.
  \ee
  The sequence is normalized and clearly the coefficients $\set{\frac{\lambda_n}{c_n(y)}-\mu^2_{k_n},\,\,n=1,2,\ldots}$ are uniformly bounded, hence $\set{u_n}_{n=1}^\infty$ is uniformly bounded in the Sobolev space $H^2(0,H).$ From the Sobolev embedding theorem and ~\eqref{equntozero} we infer
  \be\label{equntozerou'}
  \lim\limits_{n\to\infty} \int_a^b\Big(u_n(y;\lambda_n,k_n)^2+u_n'(y;\lambda_n,k_n)^2\Big)c(y)^{-1}dy=0.
  \ee
  This is a contradiction to the fact (see ~\eqref{equation-definition-cybv-unif})
  $$u_n(y;\lambda_n,k_n)^2+u_n'(y;\lambda_n,k_n)^2\geq \mathfrak{r}_V^{2},\quad n=1,2,\ldots$$

\end{proof}
\begin{proof}[\textbf{Proof of Lemma ~\ref{lem-eigenfunction:1-new}}]
We use the direct sum representation ~\eqref{equation-spectralstructure:1} both for the operator $A$ and the operators $A^{(n)}.$ Since the eigenfunctions $\set{\phi_k(x')}$ do not depend on the index $n,$ the reduced operators  $A_k^{(n)}$ (see ~\eqref{eqAkofy}) are given by
$$A_{k}^{(n)}=c^{(n)}(y)\big( \mu_{k}^{2} - \frac{d^{2}}{dy^{2}}\big), k = 1,2,\ldots,\mbox{ with }D(A_k^{(n)})=H^2(0,H)\cap H^1_0(0,H).
$$
Fix $k\in \mathbb{N}^*$ so that $\lambda$ is an eigenvalue of $A_k.$ The corresponding (reduced) eigenfunction $u(\ylambk)$ satisfies the equation (see ~\eqref{eqnonguide})
$$
 u''(\ylambk) + (\frac{\lambda}{c(y)}-\mu^{2}_{k})u(\ylambk)= 0, \quad u(\ylambk)(0) = u(\ylambk)(H) = 0.
$$
Let $B(\lambda,\delta)\subseteq  \mathbb{C}$ be the disk of radius $\delta$ centered at $\lambda$
   and consider the following linear initial value problem, with a complex parameter $z\in {\overline{B(\lambda,\delta)}},$
%
 \be\label{eqwinitz-a}
  \begin{cases}w''(y;z) + (\frac{z}{c(y)}-\mu^{2}_{k})w(y;z)= 0,\\
   w(0;z) = 0,\,\,w'(0;z)=u'(0;\lambda,k).
   \end{cases}
  \ee

      For every $y\in [0,H]$ the function $w(y;z)$ is analytic as a function of $z$  ~\cite[Chapter 1, Th.8.4]{coddington} and this is true in particular for $f(z):=w(H;z).$ Note that $z$ is an eigenvalue of $A$ if and only if $f(z)=0,$
     since  if $w$ is an eigenfunction then so is $aw,$ for any $a\neq 0.$ Clearly $f(\lambda)=0.$ This is the only zero of $f$ in  $B(\lambda,\delta)$ for sufficiently small $\delta>0,$ since $\lambda$ is an isolated eigenvalue of
      $A_k.$

      By  standard formulas for zeros of analytic functions, since $\lambda$ is a simple zero,
      \be\label{eqf'flambda}
       1=\frac{1}{2\pi i}\int\limits_{|z-\lambda|=r}\frac{f'(z)}{f(z)}dz,\quad \lambda=\frac{1}{2\pi i}\int\limits_{|z-\lambda|=r}z\frac{f'(z)}{f(z)}dz.
      \ee

      Replacing in ~\eqref{eqwinitz-a} $c(y)$ by  $c^{(n)}(y)$ we obtain solutions $w^{(n)}(y;z).$ From the equation and the initial condition we infer that $\set{w^{(n)}(y;z),\,\,z\in {\overline{B(\lambda,\delta)}} }_{n=1}^\infty$ is uniformly bounded in the Sobolev space $H^2(0,H).$ Fix $z\in B(\lambda,\delta)$ and  let $\set{z^{(n)}}\subseteq B(\lambda,\delta)$ be a sequence converging to $z.$ The Rellich compactness theorem yields the existence of a subsequence $\set{w^{(n_j)}}_{j=1}^\infty$ and a limit function $\tilde{w}(y;z)$ such that
      \be\label{eqconvwnj}
      \lim\limits_{j\to\infty}w^{(n_j)}(y;z^{(n_j)})=\tilde{w}(y;z),\quad \lim\limits_{j\to\infty}w^{(n_j)'}(y;z^{(n_j)})=\tilde{w}'(y;z),\,\,\mbox{strongly in}\,\,L^2(0,H),
      \ee
        and $w^{(n_j)''}(y;z^{(n_j)})$ converges strongly to $\tilde{w}''(y;z)$ due to the equation itself. It follows that $\tilde{w}(y;z)$ satisfies   ~\eqref{eqwinitz-a} with the same initial data, so by uniqueness
        $\tilde{w}(y;z)=w(y;z).$ In particular, since  all converging  subsequences have the same limit,~\eqref{eqconvwnj} can be replaced by
         \be\label{eqconvwnn}
      \lim\limits_{n\to\infty}w^{(n)}(y;z^{(n)})=w(y;z),\quad \lim\limits_{n\to\infty}w^{(n)'}(y;z^{(n)})=w'(y;z),\,\,\mbox{strongly in}\,\,L^2(0,H),
      \ee

         Setting $f^{(n)}(z)=w^{(n)}(H;z)$ we obtain from ~\eqref{eqf'flambda}, in view of the convergence ~\eqref{eqconvwnn} and for sufficiently large $n,$
          \be\label{eqfn'fnlambda}
       1=\frac{1}{2\pi i}\int\limits_{|z-\lambda|=r}\frac{{f^{(n)}}'(z)}{f^{(n)}(z)}dz,\quad \lambda^{(n)}=\frac{1}{2\pi i}\int\limits_{|z-\lambda|=r}z\frac{{f^{(n)}}'(z)}{f^{(n)}(z)}dz,
      \ee
          where
          the real sequence $\set{\lambda^{(n)}}$ satisfies $\lim\limits_{n\to\infty}\lambda^{(n)}=\lambda.$

          In addition, $\lambda^{(n)}$ (for sufficiently large $n$) is an eigenvalue of $A_k^{(n)}$ and $w^{(n)}(y;\lambda^{(n)})$ is an associated (not necessarily normalized) eigenfunction.
          To conclude the proof of the lemma we take $$u^{(n)}(y;\lambda^{(n)},k)=
          \frac{w^{(n)}(y;\lambda^{(n)})}{\Big[\int_0^H|w^{(n)}(y;\lambda^{(n)})|^2(c^{(n)}(y))^{-1}dy\Big]^\frac12}.$$

\end{proof}


\section{\textbf{THE DIFFUSION COEFFICIENT IS PIECEWISE CONSTANT--DETAILED STUDY}}
\label{section-piecewiseconstant:1}
 There is special physical interest in the case that the diffusion coefficient is piecewise constant.  For this reason, we focus here on this case, providing detailed information for both guided and non guided waves.  Of course, in this case $c(y)$ is of bounded variation, hence the results of Theorem ~\ref{theo-general-c-:1} and Theorem ~\ref{prop-minam:1-new} are applicable. However, we get here more detailed estimates by using more direct methods.

       \textbf{Notational comment:} As in Subsection ~\ref{subsecpiececonst}, in order to keep the notational uniformity with the other sections, we retain the notation $c_m,\,c_M$ for the minimal and maximal values, respectively, of $c(y).$ Of course they coincide with some $c_j's$ but the distinction in  various estimates   will be completely clear.

   This particular case is related to optical fibers for their industrial applications in both acoustics and optics and to printed circuit boards. The simplest example of an optical fiber is the step-index fiber: the fiber has two cylindrical parts sharing the same axis: a core of radius $a$ surrounded by a ring of thickness $b$, the cladding.  A buffer and a jacket protect these two elements by surrounding them. The index of the core is $n= n_1>0$ and that of the cladding are $n=n_2<n_1.$  Our coefficient of diffusion $c$ is exactly $c=\frac{1}{n}$. According to the choices of the respective constants $a,b, n_1,n_2$ and of the used pulse, the fiber is a single-mode fiber or a multiple-mode fiber.
   The non-specialist reader interested in these modes of data transport could consult many sites\footnote{\; for example
https://en.wikipedia.org/wiki/Optical$_{-}$fiber, http://www.fiberstore.com/Single-Mode-VS.-Multimode-Fiber-Cable-aid-340.html,
https://en.wikipedia.org/wiki/Printed$_{-}$circuit$_{-}$board}.

The cross-section of a fiber is a disk that we match to our open set $\Omega:=(0,L)\times(0,H)$ which are therefore not diffeomorphic. The analogy between the step-index fiber and this subsection is achieved by taking $N=1, c_{0}<c_{1}$ in the notations that follow. A pulse being fixed, a link between our work and the properties of the optical fibers is the following one: in the both cases we reduce the question to  a one-dimensional problem by separation of variables, replacing our variable $y$  by $r,$ the distance to the center of the disk. For the fiber the problem reduces to a family of Bessel equations according to modes E, H, ...  and the chosen simplifications whereas it is (2.2) for us. This correspondence
  has certainly theoretical limits: for the Dirichlet Laplacian in a disk there are eigenfunctions associated to high eigenvalues that are concentrated close to the boundary of the disk (see \cite[Section 7.7]{grebenkov-1} , \cite{grebenkov}). As a matter of fact, the frequencies not going to infinity in applications, this influence is reduced.
\subsection{GUIDED WAVES FOR MONOTONE PIECEWISE CONSTANT $c(y)$}
\label{subsec-guided-monotone}
$\empty$\\
This self contained subsection is a special case of Section \ref{section-guidedwaves:1}  by assuming $c$ is piecewise constant and monotone increasing when $0<y<H,$ which is the common structural assumption in physical applications, in particular in studies of optical fibers. With the above-mentioned precautions we are therefore dealing with the analogous case of the graduated-index fibers when the index $n$ is piecewise constant
and we prove the existence of these specific modes that are evanescent in the cladding. In our problem we find again the same properties of concentration of energy in the first layers of $\Omega.$ This concentration is increasing when their  thickness  decreases.

We begin by listing the hypotheses in this part (see Figure \ref{image4}).\\
\vskip-1cm
\begin{minipage}[t]{170mm}
\begin{wrapfigure}{r}{7.8cm}
\vskip-1.35cm
\hspace{-1cm}\includegraphics[scale=.36]{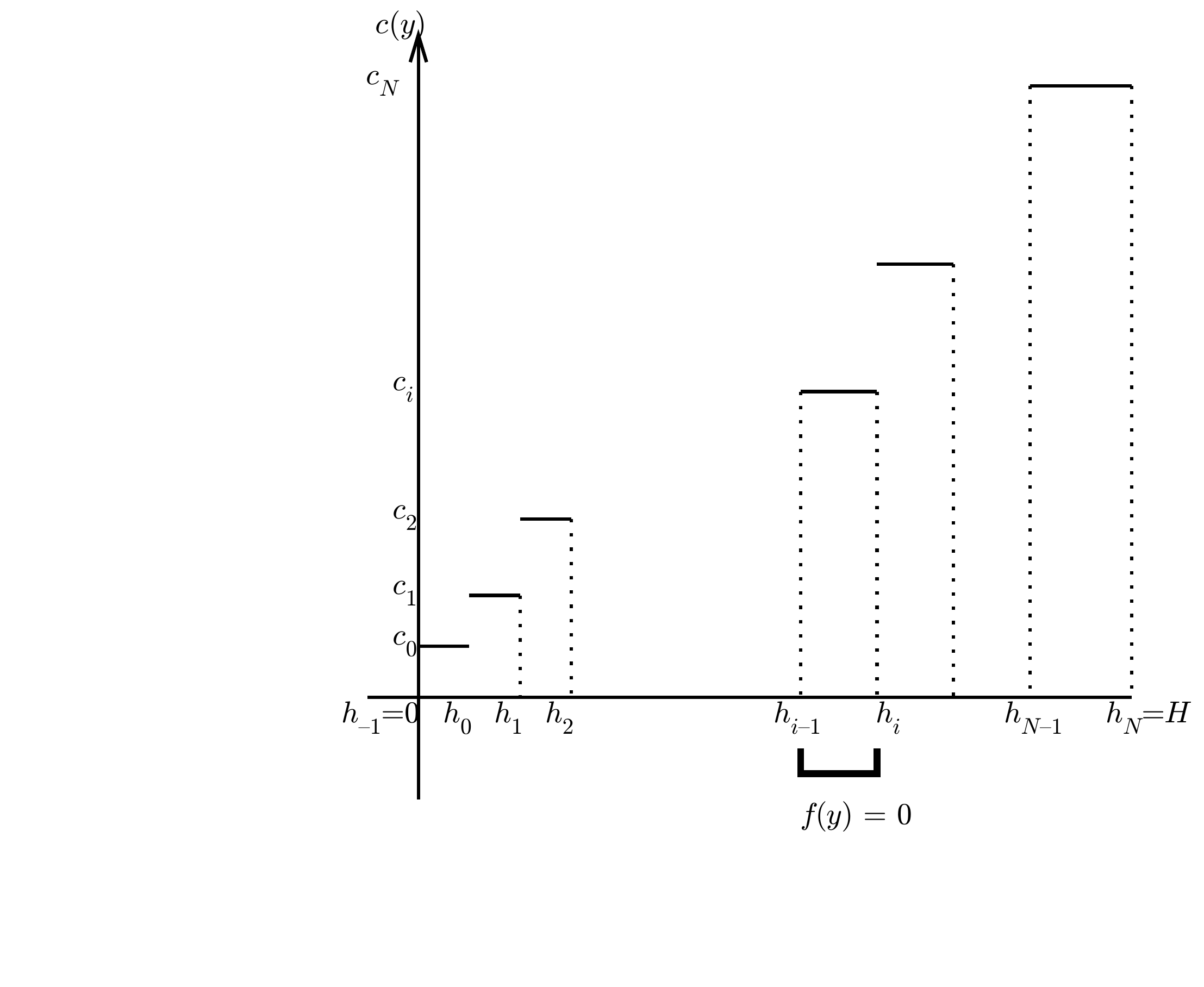}\vskip-1.4cm\caption{\label{image4}}\hfill
\vskip-.5cm
\end{wrapfigure}
\vskip.8cm
\noindent
Let $N\!\geq\! 1$, $h_{-1}\!:=0<h_0<h_1<\cdots<h_{N-1}<\!h_N:=H$,
\vskip.2cm
$c_0<c_1<\cdots<c_{N}$ and $c$ defined in $\underset{-1\leq j\leq N-1}\cup(h_j,h_{j+1})$
 \vskip.2cm
 by $c\vert_{(h_j,h_{j+1})}=c_{j+1}$ for $j=-1,\cdots,N-1.$
\vskip1cm
Let us fix $i\in \{1,\cdots, N\}$ and assume that, for a certain pair $(k,\ell),$ the eigenvalue $\lambda=\beta_{k,\ell} \in (c_{i-1}\mu_k^2,c_i\mu_k^2)$.\\
Let us denote
$ \xi^2:=\xi_{k,i}^2=2(\mu_k^2-\frac \lambda{c_i}).$\\
{\bf Recall: }$v_{\lambda}(x',y)= v_{k,\ell}(x',y) = \phi_k(x')u_{\lambda,k}(y)$
 \end{minipage}
 \vskip1cm
 \noindent
Then, for an eigenfunction $u_{\lambda,k}$ associated to $\beta_{k,\ell}$, the function $w = u_{\lambda,k}^{2}$ is a solution of \eqref{pbg} with $f$ given by
\begin{equation}\label{fp}
f(y)=\left\{\begin{array}{lll}2\lambda\frac{c_i-c_j}{c_ic_j} w(y),\quad y\in (h_{j-1}, h_{j}), \,\, j\in \{0,\cdots, N\}\setminus \{i\},\\
0, \quad y\in (h_{i-1},h_i),
\end{array}\right.
\end{equation}
and $g$ as in \eqref{pbg}.
\begin{prop}[Upper pointwise bounds]
\label{hl}
There exists $C>0$ independent of $k, \lambda, i$ such that  if $w(y)$ is a solution to \eqref{pbg}, $\delta\in [0,H-h_{i-1})$, then
 \begin{equation*}\label{vl}
w(h_{i-1}+\delta)\leq \frac \lambda \xi\Big( \underset{0\leq j\leq i-1}{\sum} \frac{c_{i}-c_{j}}{c_{i}c_{j}} e^{-\xi (h_{i-1}+\delta-h_{j})}\int_{h_{j-1}}^{h_{j}}w(y)\,dy \Big)\leq C\frac \lambda \xi e^{-\xi \delta}\int_0^{h_{i-1}}w(y)\,dy.
\end{equation*}
\end{prop}

\begin{proof}
Using the Green kernel $G$ (see \eqref{green}) we can write
$$w(y)= -\int_0^HG(y,y';\xi)\,f(y')\,dy'+ 2\int_0^HG(y,y';\xi)u'_{\lambda,k}(y')^2\,dy', \forall y\in (0,H).$$
 Then the definitions of $G$ and $f$ (see \eqref{fp}), combined with the monotonicity of $c$ and $y<x = h_{i-1}+\delta$  if $y\in [0,h_{i-1}]$, imply
$$w(h_{i-1}+\delta)\leq  \frac{2\lambda}{\xi} \underset{0\leq j\leq i-1}\sum \frac{c_i-c_j}{c_ic_j}  \frac {\sinh(\xi (H-h_{i-1}-\delta))}{\sinh(\xi H)} \int_{h_{j - 1}}^{h_{j}} \sinh(\xi y')\,w(y')\,dy',$$
where we used $\frac {\sinh(\xi (H-h_{i-1}-\delta))}{\sinh(\xi H)}\leq e^{-\xi (h_{i-1}+\delta)},\;\underset{(h_{j-1},h_j)}\sup \sinh(\xi y')\leq \frac{1}{2} e^{\xi h_j}.$
\end{proof}
\begin{prop}[Lower pointwise bounds]
\label{prop2}

Let $\delta\in [0,H-h_{i-1})$,  $[a,b]\subseteq [h_{i-1}+\delta ,H]$, then any $w$ solution of \eqref{pbg} verifies
\begin{equation*}\label{v4i} w(h_{i-1}+\delta )\geq \xi^2 (a-h_{i-1}-\delta)\int_a^b w(y)\,dy.\end{equation*}

\end{prop}
\begin{proof}
Taking into account that $w(H)=w'(H)=0$, $w$ integrating twice \eqref{pbg} and taking into account that $(u'_{k,\lambda})^2\geq0,$ we obtain
$$w(h_{i-1}+\delta)\geq 2 \int_{h_{i-1}+\delta}^H  \int_y^H \xi^{2}
w(z)\,dz\,dy\geq  \xi^2 (a-h_{i-1}-\delta)\int_a^b w(y)\,dy.$$
\end{proof}
\begin{thm}[Concentration in $\lbrack 0,h_{i-1}\rbrack$]
\label{theo-concentration-piecewisecase:1}
Let $\delta\in [0,H-h_{i-1})$ and $[a,b]\subseteq (h_{i-1}+\delta,H\rbrack.$ Then there exists $C>0$ such that for all  $k\geq 1$, $i\in \{1,\cdots,n\}$, all  $\lambda = \beta_{k,\ell}\in (c_{i-1}\mu_k^2,c_i\mu_k^2)$, the eigenfunction $u_{\lambda,k},$ satisfies
\begin{equation}\label{v6}
\int_a^bu_{\lambda,k}^2(y)\,dy\leq C\frac{\lambda}{\xi^3}\underset{0\leq j\leq i-1}\sum e^{-\xi(h_{i-1}+\delta-h_{j})}\int_{h_{j-1}}^{h_{j}}u_{\lambda,k}^2(y)\,dy.
\end{equation}
\end{thm}
\begin{proof}
It suffices to apply Propositions~\ref{hl} and \ref{prop2} with $C = \frac{1}{a-h_{i-1}-\delta}\underset{0\leq j\leq i-1}{\max} \frac{c_{i} - c_{j}}{c_{i} c_{j}}.$
\end{proof}

While the estimate ~\eqref{v6} is more precise than ~\eqref{v8G}, it implies the same type of exponential decay  of the eigenfunctions
$v_{\lambda}(x) = \phi_k(x')u_{\lambda,k}(y)$ as follows.

\begin{cor}
\label{coro-concentration-piecewisecase:1}
Let $\delta\in [0,H-h_{i-1})$ and $[a,b]\subset (h_{i-1}+\delta,H\rbrack.$ Then for all  $k\geq 1$, $i\in \{1,\cdots,n\}$, all  eigenvalues $\lambda=\beta_{k,\ell}\in(c_{i-1}\mu_k^2,c_i\mu_k^2)$ and each eigenfunction $v_{\lambda},$ one has (see Definition \ref{defn-layer:1})
\begin{equation}\label{v8}
\int_{\Omega_{a,b}}v_{\lambda}^2(x)\,dx\leq  \frac{1}{a-h_{i-1}-\delta}\frac{c_{i}-c_{0}}{c_{i} c_{0}}\frac{\lambda}{\xi^3}e^{-\xi \delta}\int_{\Omega_{0,h_{i-1}}}v_{\lambda}^2(x)\,dx.
\end{equation}
\end{cor}
\noindent
The reader may wonder if the concentration takes place only in the layer $\Omega_{h_{i-2}, h_{i-1}}$, the concentration in the other layers $\Omega_{h_{ip}, h_{ip+1}}, 0\leq p\leq i-3,$ becoming negligible when $k\to\infty$ ? Theorem \ref{annexe-theoreme-2sauts-zone(I):1} is a counter-example with the eigenfunctions $v_{\lambda} =\phi_{k}(x')u_{\lambda,k}(y)$ where
\begin{equation}
\label{equation-2sauts-zone(I)-App:1}
u_{\lambda,k}(y) =\left\lbrace\begin{array}{lll}
 a_{0}\sin(\xi_{0}y),&0<y<h_{0}, &\xi_{0} = \sqrt{\frac{\lambda}{c_{0}}-\mu_{k}^{2}},\\
a_{1}\sin(\xi_{1}y) + b_{1}\cos(\xi_{1}y)&h_{0}<y<h_{1}, &\xi_{1} = \sqrt{\frac{\lambda}{c_{1}}-\mu_{k}^{2}},\\
 a_{2}\sinh(\xi_{2}(H-y)),& h_{1}<y<H, &\xi_{2} = \sqrt{\mu_{k}^{2}- \frac{\lambda}{c_{2}}}.
\end{array}\right.
\end{equation}
\begin{thm}
\label{annexe-theoreme-2sauts-zone(I):1}
 Let  $c(y)$ be piecewise constant, taking three increasing values $c_0 < c_1 < c_2  $ and let $\set{v_{\lambda_n}}_{n=1}^\infty$ be a sequence of eigenfunctions associated with the eigenvalues $\lambda_{n}=\beta_{k_{n},l_{n}}$ satisfying $(c_{1}+\varepsilon)\mu_{k_{n}}^{2} <\beta_{k_{n},\ell_{n}}< (c_{2}-\varepsilon)\mu^{2}_{k_{n}}$ where $0<\varepsilon<\frac{c_{2}-c_{1}}{2}$. Then the  $L^{2}$ norms of the eigenfunctions $v_{\lambda_n}$ concentrate in $\Omega_{0,h_{0}}\cup\Omega_{h_{0},h_{1}}$ (the lower two layers) when $n\to\infty$ as follows.
\begin{eqnarray}
\label{equation-2sauts-zone(I):2}
h_{1}\leq a<b<H & \Longrightarrow & \int_{\Omega_{a,b}}\vert v_{\lambda_{n}}(x)\vert^{2}{\rm d}x\approxeq\frac{e^{-2\xi_{2}(a-h_{1})}}{\xi_{2}}\int_{\Omega_{0,h_{0}}}\vert v_{\lambda_{n}}(x)\vert^{2}{\rm d}x,\\
\label{equation-2sauts-zone(I):3}
\int_{\Omega_{0,h_{0}}}\vert v_{\lambda_{n}}(x)\vert^{2}{\rm d}x &\approxeq & \int_{\Omega_{h_{0},h_{1}}}\vert v_{\lambda_{n}}(x)\vert^{2}{\rm d}x.
\end{eqnarray}
\end{thm}
\begin{proof}
The following transmission conditions hold
\begin{eqnarray*}
\label{equation-annexe-detailcalculs-zone(1)-App:1}
a_{1}\sin(\xi_{1}h_{0}) + b_{1}\cos(\xi_{1}h_{0}) & = & a_{0}\sin(\xi_{0}h_{0})\nonumber\\
a_{1}\xi_{1}\cos(\xi_{1}h_{0}) - \xi_{1}b_{1}\sin(\xi_{1}h_{0}) & = & a_{0}\xi_{0}\cos(\xi_{0}h_{0})\nonumber\\
a_{1}\sin(\xi_{1}h_{1}) + b_{1}\cos(\xi_{1}h_{1}) & = & a_{2}\sinh(\xi_{2}(H-h_{1}))\nonumber\\
a_{1}\xi_{1}\cos(\xi_{1}h_{1}) - \xi_{1}b_{1}\sin(\xi_{1}h_{1}) & = &- a_{2}\xi_{2}\cosh(\xi_{2}(H-h_{1})).
\end{eqnarray*}
A straightforward  calculation gives
\begin{eqnarray}
\label{equation-annexe-detailcalculs-zone(1)-App:1'}
a_{1}^{2} + b_{1}^{2}  & = & a_{0}^{2}\left( \sin^{2}(\xi_{0}h_{0}) + (\frac{\xi_{0}}{\xi_{1}})^{2}\cos^{2}(\xi_{0}h_{0})\right),\\
a_{1}^{2} + b_{1}^{2} & = & a_{2}^{2}\left( \sinh^{2}(\xi_{2}(H-h_{1})) + (\frac{\xi_{2}}{\xi_{1}})^{2}\cosh^{2}(\xi_{2}(H-h_{1}))\right),\nonumber\\
a_{2}^{2} & = & a_{0}^{2}\frac{ \sin^{2}(\xi_{0}h_{0}) + (\frac{\xi_{0}}{\xi_{1}})^{2}\cos^{2}(\xi_{0}h_{0})}{\sinh^{2}(\xi_{2}(H-h_{1})) + (\frac{\xi_{2}}{\xi_{1}})^{2}\cosh^{2}(\xi_{2}(H-h_{1}))}.
\end{eqnarray}
As we assume $(c_{1} + \varepsilon)\mu_{k}^{2}<\lambda<(c_{2}-\varepsilon)\mu_{k}^{2}$ we have $\xi_{2}\to\infty$ as well as
\begin{eqnarray}
\label{equation-annexe-detailcalculs-zone(1)-App:3}
\frac{c_{1}\varepsilon}{c_{2}(c_{2}-c_{1}-\varepsilon)} &\leq \left( \frac{\xi_{2}}{\xi_{1}}\right)^{2} \leq \frac{c_{1}(c_{2}-c_{1}-\varepsilon)}{c_{2}\varepsilon}\mbox{ and }
\frac{c_{1}(c_{2}-c_{0}-\varepsilon)}{c_{0}(c_{2}-c_{1}-\varepsilon)} &\leq \left( \frac{\xi_{0}}{\xi_{1}}\right)^{2} \leq \frac{c_{1}(c_{1}-c_{0}+\varepsilon)}{c_{0}\varepsilon}.
\end{eqnarray}
So, for each sequence of distinct eigenvalues $(\lambda_{k})_{k}, \lambda_{k}= \beta_{k, l_{k}}$ the coefficients  $a_{1}^{2} + b_{1}^{2}$  and $a_{0}^{2}$ are comparable from \eqref{equation-annexe-detailcalculs-zone(1)-App:1'}. Moreover, there exist two constants $M_{1},M_{2}>0$, depending on $\varepsilon,$ such that
\begin{equation}
\label{equation-annexe-detailcalculs-zone(1)-App:4}
M_{1} e^{-2\xi_{2}(H-h_{1})}\leq( \frac{a_{2}}{a_{0}})^{2}\leq M_{2}e^{-2\xi_{2}(H-h_{1})},
\end{equation}
i.e. $(a_{2}/a_{0})^{2}\approxeq  e^{-2\xi_{2}(H-h_{1})}$. This concludes the proof.
\end{proof}
  It remains to be proved that the condition in Theorem ~\ref{annexe-theoreme-2sauts-zone(I):1} is not void, namely, that for each index $i$ and each $k$ sufficiently large there exists at least one eigenvalue $\beta_{k,\ell}$ located in $(c_{i-1}\mu_k^2,c_i\mu_k^2).$ This is proved in the following theorem subject to an additional hypothesis which restricts the class of operators considered.
Note that Theorem \ref{theo-existence-eigenvalues-general case:1} does not guarantee that eigenvalues $\beta_{k,\ell}$ are included in the interval $  (c_{i-1}\mu_k^2,c_i\mu_k^2).$ The additional hypothesis mentioned above is sufficient to obtain this fact.
\begin{thm}
\label{theoreme-2sauts-nombrevaleurspropres:1}
 For $\eps>0$ sufficiently small, a sufficient condition for the existence of an infinite sequence of eigenvalues
 \be\label{eqevsinc1c2}\set{\beta_{k, \ell_{k}}}_{k=1}^\infty\subseteq ((c_{1}+\varepsilon) \mu^{2}_{k},(c_{2}-\varepsilon)\mu^{2}_{k})
 \ee
  of the operator  $A = -\tilde{c}\,\Delta$  is
\begin{equation}
\label{equation-2sauts-nombrevaleurspropres-zone(I):1}
\sqrt{\frac{c_{1}(c_{1}-c_{0})}{c_{0}(c_{2}-c_{1})}} <  \frac{h_{1}-h_{0}}{H-h_{0}}.
\end{equation}
\end{thm}
Note that the inequality \eqref{equation-2sauts-nombrevaleurspropres-zone(I):1} requires $c_{1}^{2}<c_{0}c_{2}.$
For the proof, see Appendix \ref{appendix-auxiliary-results-pwc:1}.

\noindent
\begin{rem}$\empty$
\begin{enumerate}
\item In this Subsection, we  have considered a monotone increasing function $c.$ So, the concentration takes place in the union of layers such that $\left\{(x',y); y<\inf_z\{\frac{\lambda}{\mu^{2}_{k}}+\varepsilon <c(z)\}\right\}$ for a guided eigenvalue $\lambda= \beta_{k,\ell}$.
\item  Wilcox \cite{Wil:1} studied similar stratified media but the operator $-c(y)\Delta$ acted in $\mathbb{R}^{n+1}$ or $\mathbb{R}^{n+1}_{+}$ whence the point spectrum was empty. Idem in \cite{DerGui:2} where eigenvalues could appear by perturbing the coefficient $c.$ The concentration in a layer needed a local
 minimum of $c$ in this layer.
\item In the last page of \cite{BBD:2}, we pointed out that, if $\xi_{1}\to 0$, the concentration could take place in the layer $\Omega_{1}$ but it is not clear that the phenomenon could actually take place.
\end{enumerate}
\end{rem}

\subsection{\textbf{NON-GUIDED WAVES FOR GENERAL PIECEWISE CONSTANT $c(y)$}}
\label{subsection-piecewise-constant-nonguided:1}
$\empty$\\
For the non-guided waves we can apply Proposition ~\ref{propcyconst} and its proof. That proof was  technically involved, as we looked for estimates depending only on total variation, independent of the number $N$ of intervals. Since we want this section to be independent of the preceding ones, we give here the statement and a simplified  proof, addressing directly the non-concentration for a class of eigenfunctions.

Our setup here is identical to that of Subsection ~\ref{subsec-guided-monotone} with the exception that no monotonicity assumption is imposed on the $c_j's:$

 There exist
 $0=h_{-1}<h_0<h_1<\ldots<h_N=H,$
           and positive constants $c_0,c_1,\ldots,c_N$ so that
           \be\label{eqhjcj}
            c(y)=c_{j+1},\,\,\,y\in(h_j,h_{j+1}),\,\,j=-1,0,\ldots,N-1.
           \ee

 Recall ~\eqref{eqAceps} that $$\mathscr{A}^{c}_\eps=\set{(\mu_k,\lambda),\,\,\,\lambda\geq (c_{M}+\eps)\mu_k^2}.$$

 \begin{prop}
 \label{prop-minam:2}
 Consider the set of normalized solutions $u(\ylambk)$ to ~\eqref{eqnonguide}, where
 $(\mu_k,\lambda)\in \mathscr{A}^{c}_\eps.$ Then there are constants
 $\lambda_0>0,\,d>0,$ depending on $\eps,\,c_{M}, b-a,$ such that
 \be\label{eqintu2ab}
       \int_a^b u(\ylambk)^2dy\geq d,\quad  (\mu_k,\lambda)\in \mathscr{A}^{c}_\eps, \lambda>\lambda_0.
\ee
 In order to replace $\int_a^b u(y;\lambda,k)^2 dy$ by $\int_{\omega}v_{\lambda}(x)^{2} dx$ we must keep the same restriction as in Theorem \ref{theo-C1-noguided}.
 \end{prop}

         \begin{proof}
         We may assume that $c_j\neq c_{j+1},\,\,j=0,\ldots,N-1.$
           The function $p$ of ~\eqref{eqnonguide} is given by
           \be\label{eqhjpj}
            p(\ylambk)=p_{j+1}=
            \frac{\lambda}{\mu^2_{k}c_{j+1}}-1,\,\,\,y\in(h_j,h_{j+1}),\,\,j=-1,0,\ldots,N-1,
           \ee
           so that
           \be\label{eqpgeepscmax}
           p(\ylambk)\geq \frac{\eps}{c_{M}},\quad y\in [0,H].
                      \ee

Consider an interval $I_j=(h_j,h_{j+1}).$ The normalized solution $u(\ylambk)$ to ~\eqref{eqnonguide} in $I_j$ is given by
         \be\label{equyinIj}
         u(\ylambk)=\beta_j\sin(\mu_{k}\sqrt{p_{j+1}}(y-z^{(j)}_{0})),\quad y\in I_j.
         \ee
\noindent
            where we denote a zero of $u(\ylambk)$ in $I_j$ by $z^{(j)}_{0}$. Observe that (by the comparison principle) there exists a constant $\lambda_0>0,$ such that if $\lambda>\lambda_0$ then there are at least two zeros $h_j<z^{(j)}_{0}<z^{(j)}_{1}<h_{j+1}$ of $u(\ylambk)$ in every interval $I_j.$ Furthermore, we can assume that
            $$z^{(j)}_{1}-z^{(j)}_{0}\geq\frac12(h_{j+1}-h_j),\quad j=-1,0,\ldots,N-1.$$

            Suppose now that for some $\eta>0$ depending only on $c_{m},\,c_{M},\,\eps,$ we have
            \be\label{eqlimitbeta}
 \min\limits_{-1\leq j\leq N-1}|\beta_j|\geq \eta.
 \ee
      Then the assertion of the proposition is established as follows:  it can be assumed that the interval $(a,b)\subseteq I_j$ for some $j,$ since it can be replaced by a non-void intersection with some $I_j.$ Furthermore we can increase (if necessary) the constants $\lambda_0$ so that
 $u(\ylambk)$ has at least two zeros $a<y_1<y_2<b,$ with $y_2-y_1>\frac{b-a}{2}$ when $\lambda>\lambda_0.$

 Next in view of ~\eqref{eqpgeepscmax} , ~\eqref{equyinIj} and ~\eqref{eqlimitbeta}
\be\label{eqintab}
\int_a^bu(\ylambk)^{2}dy\geq \int_{y_1}^{y_2}u(\ylambk)^{2}dy\geq\frac14(b-a)\beta_j^2\geq\frac14(b-a)\eta^2.
\ee
    Thus we need to establish ~\eqref{eqlimitbeta}.

            We claim that there exists a constant $C>1,$   independent of $\mu_k,\,\lambda,$
         such that
            \be\label{eqsumbeta'}C^{-1}\leq\suml_{j=-1}^{N-1}\beta_j^2\leq C.
            \ee
            Indeed, $|\beta_j|\geq |u(\ylambk)|$ for $y\in I_j$ hence
            $$\aligned\suml_{j=-1}^{N-1}\beta_j^2(h_{j+1}-h_j)\geq \int_{0}^{H}(u(\ylambk))^{2} dy \\\geq
            c_{m}\int_{0}^{H}(u(\ylambk))^{2} c(y)^{-1}dy=c_{m}.\endaligned$$
            Thus
            \be\label{eqbetaj2inf}
            \suml_{j=-1}^{N-1}\beta_j^2\geq\frac{c_{m}}{\max \vert h_{j+1}-h_{j}\vert}.
            \ee

 On the other hand, use ~\eqref{equyinIj} between two zeros in $I_j,$
 \be\label{eqbetaj2sup}\int_{0}^{H}(u(\ylambk))^{2} dy\geq \suml_{j=-1}^{N-1}\int_{z^{(j)}_{0}}^{z^{(j)}_{1}}(u(\ylambk))^{2} dy\geq \frac12\suml_{j=-1}^{N-1}(z^{(j)}_{1}-z^{(j)}_{0})\beta_j^2\geq\frac{\min (h_{j+1}-h_j)}{4}\suml_{j=-1}^{N-1}\beta_j^2.
 \ee
   Combining ~\eqref{eqbetaj2inf} and ~\eqref{eqbetaj2sup} we obtain ~\eqref{eqsumbeta'}.

 Next we claim that
 \be\label{eqclmratiobeta}
 \mbox{The ratios}\,\,\Big|\frac{\beta_{j+1}}{\beta_j}\Big|,\,\,\Big|\frac{\beta_{j}}{\beta_{j+1}}\Big|,\,\,j=-1,0,\ldots
  N-2,\,\,\mbox{are uniformly bounded for }\,\, \lambda>\lambda_0.
 \ee

 To prove this claim, we set for any $-1\leq j\leq N-2$

         $$A_{j}=\mu_k\sqrt{p_{j+1}}(h_{j+1}-z^{(j)}_{0}),\,\,\,
         B_{j+1}=\mu_k\sqrt{p_{j+2}}(h_{j+1}-z_0^{(j+1}).$$
         The continuity of  $u(y;\lambda,k)$ and $u'(y;\lambda,k)$ at $h_{j+1}$ implies that
         \begin{eqnarray*}
          \label{eqcontumhj:1}
         \beta_{j}\sin(A_{j}) &= &
         \beta_{j+1}\sin(B_{j+1}),\\
         \label{eqcontumhj:2}
         \beta_{j}\sqrt{p_{j+1}}\cos(A_{j}) &= &
         \beta_{j+1}\sqrt{p_{j+2}}\cos(B_{j+1}).
         \end{eqnarray*}
      Thus,
      \be\label{eqratiosq}\aligned
      \frac{\beta_j^2}{\beta_{j+1}^2}=\sin^2B_{j+1}+\frac{p_{j+2}}{p_{j+1}}\cos^2B_{j+1},\\
      \frac{\beta_{j+1}^2}{\beta_{j}^2}=\frac{1}{\sin^2B_{j+1}+\frac{p_{j+2}}{p_j+1}\cos^2B_{j+1}}.
      \endaligned\ee
      Since the ratios $\frac{p_{j+2}}{p_{j+1}},\,\,\frac{p_{j+1}}{p_{j+2}}$ are uniformly bounded for $\lambda>\lambda_0$  (as readily seen from ~\eqref{eqhjpj}) the claim ~\eqref{eqclmratiobeta} follows.

 Combining  ~\eqref{eqclmratiobeta} and ~\eqref{eqsumbeta'} we obtain ~\eqref{eqlimitbeta}.
   \end{proof}
   In the Introduction (Remark ~\ref{rem-uniform}), we noted the fact that the coefficient $c(y)$ is often only approximately known. It is reflected in Proposition ~\ref{prop-minam:2}:  if we equip $\mathscr{K}_{PC}:= \{c(y) \mbox{ piecewise constant function};\,\, 0<c_m\leq c(y)\leq c_M\}$ with the uniform norm $\Vert c\Vert_{\infty} = \max c(y),$ the map $\mathscr{K}_{PC}\ni c\to (\lambda_0, d)$ is continuous.

\appendix

\section{{\bf PROOF OF THEOREM ~\ref{theoreme-2sauts-nombrevaleurspropres:1}} }
\label{appendix-auxiliary-results-pwc:1}
$\empty$\\


The proof  uses two steps for fixed $k,$ which circumvent the issue of multiplicity of eigenvalues of $-\Delta_{x'}$ in $\Omega'.$\\ Recall that $A_k$ is given in ~\eqref{eqAkofy}.

{\bf Step 1 :  Existence of eigenvalues of $A_{k}$ between $c_{0}\mu^{2}_k$ and   $(c_2-\varepsilon)\mu_k^{2}$.}\\
  We use $[r]$ to designate the largest integer below $r.$\\
   Let $(0,h_0)$  be divided into  $n=\big\lbrack\frac{h_{0}}{\pi}\sqrt{\frac{c_{1}-c_{0}}{c_{0}}}\mu_{k}\big\rbrack$ subintervals of equal length $h= \frac{h_{0}}{n},$ and  let$(h_{0},h_{1})$ be divided into $N = \big\lbrack\frac{h_{1}-h_{0}}{\pi}\sqrt{\frac{c_{2}-\varepsilon-c_{1}}{c_{1}}}\mu_{k}\big\rbrack$ subintervals of equal length $\mathfrak{h}= \frac{h_{1}-h_{0}}{N}.$   It follows readily that $c_0((\frac{\pi}{h})^2 + \mu_k^2)\leq c_1 \mu_k^2$ and $c_{1} ((\frac{\pi}{\mathfrak{h}})^{2} + \mu^{2}_{k}) \leq (c_{2}-\varepsilon)\mu^{2}.$

   Define a set of $n+N$  functions $\set{\varphi_{p}}_{p=1}^{n+N}\subseteq H_{0}^{1}(0,H)$  by:
\begin{equation*}
\label{equation-2sauts-nombrevaleurspropres-zone(I):2}
\begin{array}{ll}
\varphi_{p}(y)= \left\lbrace\begin{array}{l}
\sqrt{\frac{2}{h}}\sin\left(\frac{\pi}{h}(y-(p-1)h))\right)\mbox{ if }(p-1)h<y<ph,\\
0 \mbox{ otherwise},
\end{array}
\right. &\mbox{ if } 1\leq p\leq n,\\
\varphi_{p}(y)=\left\lbrace\begin{array}{l}
\sqrt{\frac{2}{\mathfrak{h}}}\sin\left(\frac{\pi}{\mathfrak{h}}(y-(h_{0}+(p-1)\mathfrak{h}))\right) \mbox{ if }  h_{0}+(p-1)\mathfrak{h}<y<h_{0} +p\mathfrak{h},\\
0 \mbox{ otherwise},
\end{array}
\right.& \mbox{ if } n + 1\leq p\leq n+N.
\end{array}
\end{equation*}
  The functions $\set{\varphi_{p}}_{p=1}^{n+N}$ are normalized (in $L^2$) and pairwise orthogonal. Furthermore, the restrictions of $\varphi_{p}$ to $ ((p-1)h, ph),\, 1\leq p\leq n$ (resp. $n+1\leq p\leq n+N$) are  the first Dirichlet-Laplacian eigenfunctions on their supporting subintervals, associated with the eigenvalue $(\frac{\pi}{h})^{2}$ (resp. $(\frac{\pi}{\mathfrak{h}})^{2}$).


 In particular, for any linear combinations we have
\be\label{eqlincomb}\aligned
\varphi(y)=\suml_{p=1}^n \gamma_p\varphi_p(y)\Rightarrow\Rightarrow\int_{0}^{H}[(\varphi')^{2} + \mu^{2}_{k}\varphi^{2}]{\rm d}y = c_{0}((\frac{\pi}{h})^{2} + \mu^{2}_{k})\int_{0}^{H} c^{-1}\varphi^{2}{\rm d}y,\\
\varphi(y)=\suml_{p=n+1}^N \gamma_p\varphi_p(y)\Rightarrow\Rightarrow\int_{0}^{H}[(\varphi')^{2} + \mu^{2}_{k}\varphi^{2}]{\rm d}y = c_{1}((\frac{\pi}{\mathfrak{h}})^{2} + \mu^{2}_{k})\int_{0}^{H} c^{-1}\varphi^{2}{\rm d}y.
\endaligned\ee

Now we use the following $\max-\min$ principle to evaluate the $r-th$ eigenvalue of $A_{k}$  ~\cite[Ch. XIII.9.D2, pp.1543-1544)]{Dun-Schw:1}:  if $\mathcal{S}_{(r)}:=\{\mathfrak{H}_{r}\subseteq H_{0}^{1}(0,H),\quad \mathfrak{H}_{r} \mbox{ is a linear space with } \dim \mathfrak{H}_{r} = r\},$ then the $r-$th eigenvalue
 of $A_{k}$  (in the weighted space $L^{2}((0,H), c(y)^{-1}dy)$)  is given by
\begin{equation}
\label{equation-2sauts-nombrevaleurspropres-zone(I):4}
\beta_{k,r} = \inf_{\mathfrak{H}_{r}\in \mathcal{S}_{(r)}}\sup_{v\in \mathfrak{H}_{r}, v\not=0} \frac{a_{k}(v,v)}{(v,v)},\quad \mbox{ where }a_{k}(v,v)=\int_{0}^{H} (v')^{2} + \mu_{k}^{2}v^{2})\,dy,\,\,(v,v)=\int_{0}^{H} v^{2}c(y)^{-1}\,dy.
\end{equation}
 Let $r=n$ and take $\mathfrak{H}_{n}=V:=span \set{\varphi_p}_{p=1}^n.$

Since
$$\sup_{v\in V, v\not=0} \frac{a_{k}(v,v)}{(v,v)}=c_{0} ((\frac{\pi}{h})^{2} + \mu_k^{2}),$$
  it follows from the definition of $h$  that
 $$\beta_{k,n}\leq c_{0} ((\frac{\pi}{h})^{2} + \mu_k^{2})\leq c_{1}\mu_k^{2}.$$  

  We conclude that \textbf{ there are at least $n$ eigenvalues less than or equal to $c_{1}\mu_k^{2}.$}

 Applying a similar argument   to the full set $\set{\varphi_{p}}_{p=1}^{n+N}$ we infer that $$\beta_{k,n+N}\leq\max\{c_{0} ((\frac{\pi}{h})^{2} + \mu_k^{2}), c_{1} ((\frac{\pi}{\mathfrak{h}})^{2} + \mu_k^{2})\} \leq (c_{2}-\varepsilon)\mu_k^{2},$$
  where in the final estimate we used the assumptions on $h,\,\mathfrak{h}.$
 \\
 As above, we conclude that, {\textbf{for a fixed $k$ there are at least $n+N$ eigenvalues of $A_k$ smaller than or equal to $(c_{2}-\varepsilon)\mu_k^{2}.$}}

{\bf Step 2 : Existence of eigenvalues of $A_{k}$ between $(c_{1}+\varepsilon)\mu^{2}_{k}$ and $(c_{2}-\varepsilon)\mu^{2}_{k}.$}\\
  Consider the self-adjoint operator $A^{0}_{k}u= -c_{0}u''+c_{0}\mu_k^{2} u$ acting in $L^{2}(0,H; c_0^{-1}dy)$ subject to Dirichlet boundary conditions.  It is associated to the variational form $a_k^0(v,w) = \int_0^H (v'(y)\overline{w'(y)} + \mu_k^2 v(y)\overline{w(y)})dy, v,w\in H_0^1(0,H).$  Let $\set{\beta^{0}_{k,p}= c_{0}\frac{\pi^{2}}{H^{2}}p^{2} + c_{0}\mu_k^{2}}_{p=1}^{\infty}$ be the nondecreasing sequence of its eigenvalues.  Using ~\eqref{equation-2sauts-nombrevaleurspropres-zone(I):4} the comparison with  the eigenvalues  of $A_{k}$ is straightforward since the  forms $a_k,\,a_k^0$ are  identical.   We conclude that  $$\beta^{0}_{k,p}\leq \beta_{k,p},\,\,p=1,2,\ldots.$$
  The explicit expression of $\beta^{0}_{k,p} $ entails that if $p>\frac{H}{\pi}\sqrt{\frac{c_{1}+\varepsilon - c_{0}}{c_{0}}}\mu_{k}$ then  $(c_{1}+\varepsilon)\mu^{2}_{k}<\beta^{0}_{k,p}\leq \beta_{k,p}.$
    On the other hand, for $p\leq n+ N,$ we saw at the end of Step 1 that $\beta_{k,p}\leq (c_2-\varepsilon)\mu_k^2.$ So, to obtain eigenvalues of $A_k$ in $((c_{1}+\varepsilon)\mu^{2}_{k},(c_{2}-\varepsilon)\mu^{2}_{k})$  it   suffices to find an integer $p$ such that
\begin{equation}
\label{equation-2sauts-nombrevaleurspropres-zone(I):5}
\frac{H}{\pi}\sqrt{\frac{c_{1}+\varepsilon - c_{0}}{c_{0}}}\mu_{k}  < p\leq \frac{h_{0}}{\pi}\sqrt{\frac{c_{1}-c_{0}}{c_{0}}}\mu_{k} + \frac{h_{1}-h_{0}}{\pi}\sqrt{\frac{c_{2}-\varepsilon-c_{1}}{c_{1}}}\mu_{k} -2.
\end{equation}
  Such an integer exists if \be\label{eqc0c1c2eps}\frac{h_{0}}{\pi}\sqrt{\frac{c_{1}-c_{0}}{c_{0}}}\mu_{k} + \frac{h_{1}-h_{0}}{\pi}\sqrt{\frac{c_{2}-\varepsilon-c_{1}}{c_{1}}}\mu_{k} -  \frac{H}{\pi}\sqrt{\frac{c_{1}+\varepsilon - c_{0}}{c_{0}}}\mu_{k}>3.\ee
  
   It is readily seen that ~\eqref{eqc0c1c2eps} holds for sufficiently large $\mu_k$ if
\begin{equation}
\label{equation-2sauts-nombrevaleurspropres-zone(I):6}
\sqrt{\frac{c_{1}(c_{1} + \varepsilon -c_{0})}{c_{0}(c_{2}-\varepsilon - c_{1})}}\left( 1 - \frac{h_{0}}{H}\sqrt{\frac{c_{1}-c_{0}}{c_{1} + \varepsilon -c_{0}}}\right)< \frac{h_{1} - h_{0}}{H}.
\end{equation}
 The validity of ~\eqref{equation-2sauts-nombrevaleurspropres-zone(I):6} with $\varepsilon = 0$ follows   exactly   from the condition  ~\eqref{equation-2sauts-nombrevaleurspropres-zone(I):1}. By continuity, \eqref{equation-2sauts-nombrevaleurspropres-zone(I):6}  will be satisfied for $\varepsilon<\varepsilon_{0},$ for sufficiently small $ \varepsilon_{0}>0.$ \hfill$\Box$\\

\end{document}